\documentclass[10pt,a4paper]{article}
\usepackage{amsmath}
\usepackage{amssymb}
\usepackage{amsthm}
\usepackage[pdftex]{color,graphicx}
\usepackage[utf8]{inputenc}
\usepackage[T1]{fontenc}
\usepackage[english]{babel}
\usepackage{bbm}
\usepackage[nice]{units}
\usepackage[super]{nth}

\usepackage{hyperref}
\newtheorem{tw}{Theorem}[section]
\newtheorem{lm}[tw]{Lemma}
\newtheorem{wn}[tw]{Corollary}
\newtheorem{pr}[tw]{Proposition}
\theoremstyle{definition}
\newtheorem{df}{Definition}[section]
\newtheorem{uw}{Remark}[section]

\newcommand{\oh}{\overline{h}}
\newcommand{\1}{\mathbbm{1}}
\newcommand{\xbm}{(X,{\cal B},\mu)}
\newcommand{\Z}{\mathbb{Z}}
\newcommand{\Q}{\mathbb{Q}}
\newcommand{\R}{\mathbb{R}}
\newcommand{\T}{\mathbb{T}}
\newcommand{\C}{\mathbb{C}}
\newcommand{\cT}{\mathcal{T}}
\newcommand{\cS}{\mathcal{S}}
\newcommand{\cP}{\mathcal{P}}
\newcommand{\N}{\mathbb{N}}
\newcommand{\ycn}{(Y,{\cal C},\nu)}

\newcommand{\vep}{\varepsilon}

\title{On the strong convolution singularity property}

\author{Joanna Ku\l{}aga-Przymus\footnote{Research partially supported by MNiSzW grant N N201 384834 and Marie Curie "Transfer of Knowledge" program, project MTKD-CT-2005-030042 (TODEQ).}\\{\small Faculty of Mathematics and Computer Science,}\\ {\small Nicolaus Copernicus University},\\{\small ul. Chopina 12/18, 87-100 Toru\'n, Poland}\\ {\small e-mail: joanna.kulaga@gmail.com}}

\begin{document}
\bibliographystyle{abbrv}
\maketitle
\setcounter{tocdepth}{4}

\begin{abstract}
We develop a new method for proving that a flow has the so-called strong convolution singularity property, i.e. the Gaussian system induced by its (reduced) maximal spectral type has simple spectrum. We use these methods to give examples of smooth flows on closed orientable surfaces of genus at least $2$ with a weaker property: each of their maximal spectral types $\sigma$ is such that the Gaussian system induced by $\sigma$ has simple spectrum on the so-called \nth{3} chaos (i.e. $V_\sigma^{\odot 3}$ has simple spectrum).
\end{abstract}

\footnotesize
\tableofcontents
\normalsize 

\section{Introduction}
Let $\mathcal{U}=(U_t)_{t\in\R}$ be a continuous unitary representation on a separable Hilbert space $H$. Each such representation is determined by two invariants: \emph{the (reduced) maximal spectral type} and the \emph{spectral multiplicity function}. The first of these invariants is the equivalence class $\sigma_\mathcal{U}$ of all spectral measures $\sigma_x$ with $x\in H$ which dominate all other spectral measures: $\sigma_y\ll \sigma_x$ for $y\in H$ (such dominating measures always exist and they are all mutually equivalent).\footnote{For the definition of a spectral measure see Section~\ref{se:spectralna}.} The spectral multiplicity function $M_\mathcal{U}\colon \R=\hat{\R}\to \N\cup\{\infty\}$ is measurable and defined $\sigma_\mathcal{U}$-almost everywhere.\footnote{We do not give a precise definition of the spectral multiplicity function as we will not use it later.} We say that $\mathcal{U}$ has simple spectrum when it is unitarily isomorphic to $\mathcal{V}_{\sigma_\mathcal{U}}=((V_{\sigma_\mathcal{U}})_t)_{t\in\R}\colon L^2(\R,\sigma_\mathcal{U})\to L^2(\R,\sigma_\mathcal{U})$ given by $(V_{\sigma_\mathcal{U}})_tf(x)=e^{2\pi i t x}f(x)$ (this is equivalent to $M_\mathcal{U}\equiv1$ $\sigma_\mathcal{U}$-almost everywhere).

We deal with unitary representations associated with measure-preserving flows. Given $\mathcal{T}=(T_t)_{t\in\mathbb{R}}$ acting on a standard probability Borel space $(X,\mathcal{B},\mu)$ one defines the so-called \emph{Koopman representation} $\mathcal{U}_{\cT}={(U_{T_t})}_{t\in\mathbb{R}}\colon L^2(X,\mathcal{B},\mu)\to L^2(X,\mathcal{B},\mu)$ by the formula $U_{T_t}(f)=f\circ T_t$ for all $t\in\R$. We often write $\sigma_{\cT}$ instead of $\sigma_{\mathcal{U}_{\cT}}$ and call it the (reduced) maximal spectral type of $\cT$. The properties of $\sigma_\cT$ may reflect important properties of the flow $\cT$. In case when $\sigma_\cT$ is a singular measure it is interesting to know what is the ``degree'' of its singularity. 

Let us recall two notions which will play here an important role.\footnote{Analogous properties can be easily formulated for $\Z$-actions or action of other groups.} We say that the flow $\cT$ has the property of the \emph{mutual singularity of convolution powers} if its maximal spectral type has this property, i.e. when $\sigma_\cT^{\ast n}\perp \sigma_{\cT}^{\ast m}$ for $n\neq m$. However, for the purposes of this paper a stronger property, so-called  \emph{strong convolution singularity property} (SCS\footnote{In~\cite{lemanczyk+parreau} this property was denoted by SC, however, following~\cite{lem-ency}, we prefer to write SCS instead. In this way we avoid confusion with the so-called CS properties defined in terms of singularity of the considered measure with respect to convolutions of continuous measures. For information on the relations between these notions see e.g.~\cite{1112.5545} and~\cite{MR2729082}.}) will be crucial. The flow $\cT$ is said to have the SCS property whenever $\sigma=\sigma_\cT$ is such that the unitary $\R$-representation $\bigoplus_{n\geq 1}V_\sigma^{\odot n}$ has simple spectrum. In other words the spectrum of $V_{\sigma}^{\odot n}$ for each $n\geq 1$ is simple and the maximal spectral types of $V_\sigma^{\odot n}$, i.e. $\sigma^{\ast n}$, for $n\geq 1$ are mutually singular: $\sigma^{\ast n}\perp \sigma^{\ast m}$ for $n\neq m$. It is a folklore result that the simplicity of spectra of all $V_\sigma^{\odot n}$ yields automatically the SCS property for $\sigma$ (for a proof see~\cite{1112.5545}). One way to check that $V_\sigma^{\odot n}$ has simple spectrum is to show that
\begin{multline}\label{eq:warunek1}
\mbox{$\sigma^{\ast n}$-almost all conditional measures in the disintegration of $\sigma^{\otimes n}$ over $\sigma^{\ast n}$}\\
\mbox{via the map $\R^n\ni (x_1,\dots,x_n)\mapsto x_1+\dots +x_n\in \R$}\\
\mbox{are purely atomic with $n!$ atoms.}
\end{multline}
The SCS property was shown to be generic for flows~\cite{lemanczyk+parreau}.\footnote{In case of automorphisms, see Ageev~\cite{MR1680995}.} While there is a variety of concrete examples of automorphisms enjoying the SCS property, including the Chacon automorphism \cite{MR1835446} and some mixing automorphisms \cite{MR2465595}, \cite{MR2354530}, much less is known in the case of flows. 

In a recent paper~\cite{lemanczyk+parreau} it was shown that some classes of smooth or regular flows on $\mathbb{T}^2$ enjoy the SCS property. The examples are given in terms of special flows $T^f=(T_t^f)_{t\in\mathbb{R}}$, where $T\colon \mathbb{T}\to\mathbb{T}$ is an irrational rotation $Tx=x+\alpha$ and $f\colon \mathbb{T}\to \mathbb{R}_+$ an integrable function (for a formal definition of special flow see Section~\ref{se:potokispecjalne}):
\begin{itemize}
\item[(A)]
for a fixed $f\in C^{\infty}$ different from any trigonometric polynomial for a ``generic'' $\alpha\in [0,1)$ the resulting flow $T^f$ has the SCS property,
\item[(B)]
for $f$ piecewise absolutely continuous, with the sum of jumps different from zero, and $\alpha$ with unbounded partial quotients the flow $T^f$ also has the SCS property.
\end{itemize}
Recall that the flows described in the case (A) are smooth reparametrizations of linear flows on the torus $\T^2$. The class (B) was introduced by von Neumann in~\cite{MR1503078} where it was shown there that these flows are weakly mixing. The methods developed in~\cite{lemanczyk+parreau} exploit the fact that for each flow of type (A) or (B) one can find $t_n\to \infty$ such that 
\begin{equation}\label{eq:warunek2}
U_{T_{t_n}}\to \int_\R U_{T_t}\ dP(t)
\end{equation}
in the weak topology, where $P$ is a Borel probability measure on $\R$. An additional feature of (A) and (B) is that the measure $P$ has always bounded support. Therefore, in order to show that~\eqref{eq:warunek1} holds, one investigates the properties of the Fourier transform of $P$ which is an analytic function on $\R$. In case (A) one deals with the first Bessel function which has an infinite number of zeros. In case (B) the Fourier transform is a finite combination of functions which are of the form $t\mapsto \frac{1}{t}e^{2\pi i t \Delta}$ for some $\Delta\in\R$. Then, using the location of zeros (case (A)) and the properties of ``polynomials'' (case (B)) one concludes that~\eqref{eq:warunek1} indeed holds. What is important is that in the weak closure there are several integral Markov operators: with measure $P$ and some of its ``rescalings'' $P_t$.\footnote{Symbol $P_t$ stands for the image $P$ via the map $x\mapsto tx$.} Indeed, in (A) these are rescalings $P_t$ for $t>0$ and in (B) for $t\in\Z$. 

The examples considered in the present paper are of different nature than the ones described in~\cite{lemanczyk+parreau}. We deal with some smooth flows on closed orientable surfaces of genus at least $2$. In their special flow representation the roof function is unbounded and smooth except for one singularity. Similar as in case (B), we obtain integral Markov operators given by some measure $P$ and its integer ``rescalings'' $P_m$ ($m\in\Z\setminus\{0\}$). The main difference between this case and the above-mentioned cases (A) and (B) is such that the measure $P$ and its ``rescalings'' $P_m$ ($m\in\Z\setminus\{0\}$) appearing here in the weak limits~\eqref{eq:warunek2} have unbounded supports. Moreover, we use the densities of the measures $P_m$ instead of looking at their Fourier transforms as it was done in case (B).\footnote{The measures $P_m$ in case (B) are also absolutely continuous.} Each of them vanishes on some half-line and their densities are smooth, except for one point where the right-hand side and left-hand side limits differ (one of them is equal to zero, the other one is infinite). We develop general tools which yield the SCS property under some additional assumptions on the form of the singularities and on the speed of the decay of the density functions. An important part of our argument is related to algebraic geometry. For each $n\geq 1$ we define an infinite system of equations consisting of some symmetric polynomials with coefficients determined by the density  of the measure $P$ (more precisely, by the speed of growth near the singularity point of the derivatives of the convolutions $P_{t_1}\ast\dots \ast P_{t_k}$, $t_1,\dots,t_k\in\mathbb{R}\setminus\{0\}$). The uniqueness of the solution of this system for some $n\geq 1$ implies the simplicity of spectrum for $V_\sigma^{\odot n}$, where $\sigma$ is the maximal spectral type of the flow under consideration. We apply these methods to show that in some class of smooth flows on closed orientable surfaces of genus at least $2$ the maximal spectral type $\sigma$ is such that $V_{\sigma}^{\odot 3}$ has simple spectrum.

\section{Definitions}

\subsection{Spectral theory}\label{se:spectralna}
Let $\mathcal{U}=(U_t)_{t\in\R}$ 
be a unitary $\R$-representation on a separable Hilbert space $H$.  Every such representation is determined by two invariants: the \emph{maximal spectral type} which is the equivalence class of a spectral measure $\sigma_x$ which dominates all spectral measures (for all $y\in H$ we have $\sigma_y\ll \sigma_x$) and the \emph{spectral multiplicity function} $M_\mathcal{U}\colon \hat{\mathbb{R}}=\R\to \N\cup \{\infty\}$. Recall that for $x\in H$ the measure $\sigma_x$ is the finite positive Borel measure on $\R$ whose Fourier transform $(\hat{\sigma}_x(t))_{t\in\R}$ is given by $\hat{\sigma}_x(t)=\langle U_{t}x,x\rangle$. Whenever $\mathcal{U}$ is isomorphic to $\mathcal{V}_{\sigma}=((V_{\sigma})_t)_{t\in\R}\colon L^2(\R,\sigma)\to L^2(\R,\sigma)$ given by $(V_{\sigma})_tf(x)=e^{2\pi i t x}f(x)$ for some measure $\sigma$, we say that $\mathcal{U}$ has \emph{simple spectrum}. In this case, the maximal spectral type of $\mathcal{U}$ is equal to $\sigma$.

Given a measure preserving flow $\cT=(T_t)_{t\in\R}$ on a standard probability Borel space $\xbm$ one defines a canonical unitary representation called the \emph{Koopman representation}: $\mathcal{U}=(U^\cT_t)_{t\in\R}\colon L^2\xbm \to L^2\xbm$
by 
$$
U^\cT_t(f)=f\circ T_t\text{ for }f\in L^2\xbm \text{ and }t\in\R.
$$
The spectral properties of $\mathcal{T}$ are the properties of the associated Koopman representation, e.g. the maximal spectral type of $\mathcal{T}$ is the maximal spectral type of $U_\mathcal{T}$. Since every Koopman representation has an atom at $0$, i.e. $\delta_0\ll \sigma_{U_\cT}$, we will use the notion of the \emph{reduced maximal spectral type}, i.e. $\sigma_{U_\cT|_{L^2_0(X)}}$, where $L^2_0(X)=\left\{f\in L^2(X)\colon \int_X f\ d\mu=0\right\}$.

We recall now some additional spectral properties.
\begin{df}\footnote{For $n=2$ this notion was introduced in~\cite{MR2729082}.}
Let $n\geq 1$. The measure $\sigma$ on $\R$ has the \emph{convolution singularity of order $n$} property (CS($n$)) whenever it is singular with respect to the convolution of any $n$ continuous measures on $\R$.
\end{df}
\begin{df}
The measure $\sigma$ on $\R$ has the \emph{strong convolution singularity property} whenever for each $n\geq 1$ the conditional measures in the decomposition
$$
\sigma^{\otimes n}=\int_{\R}\nu_{t}\ d\sigma^{\ast n}(t)
$$
are purely atomic and have $n!$ atoms.
\end{df}
Recall that the SCS property is equivalent to spectral simplicity of the Gaussian action associated to the reduced maximal spectral type of $U_\cT$. For more information concerning Gaussian systems we refer the reader e.g. to~\cite{MR0272042}.

\subsection{Joinings}
Let $\mathcal{T}$ and $\mathcal{S}$ be measure preserving flows on $\xbm$ and $\ycn$ respectively. By $J(\mathcal{T},\mathcal{S})$ we denote the set of all~\emph{joinings} between $\mathcal{T}$ and $\mathcal{S}$, i.e. the set of all $\mathcal{T}\times \mathcal{S}$-invariant probability measures on $(X\times Y,\mathcal{B}\otimes \mathcal{C})$, whose projections on $X$ and $Y$ are equal to $\mu$ and $\nu$ respectively. \emph{Ergodic joinings} are these joinings which are ergodic with respect to $\mathcal{T}\times \mathcal{S}$. Joinings are in one-to-one correspondence with Markov operators $\Phi\colon L^2(X,\mathcal{B},\mu)\to L^2(Y,\mathcal{C},\nu)$ satisfying the condition $\Phi\circ \mathcal{T}=\mathcal{S} \circ \Phi$. This correspondence is given by $\Phi \mapsto \lambda \in J(\cT,\cS),\ \lambda(A\times B)= \int_B \Phi(\1_A)\ d\nu$. The set of such Markov operators is endowed with the weak operator topology which allows us to view the set $J(\mathcal{T})$ of all self-joinings of $\cT$ (i.e. joinings of $\cT$ with itself) as a metrisable compact semitopological semigroup. We refer the reader to~\cite{MR1958753} for more information on the theory of joinings and e.g. to~\cite{MR1784644} for a short survey on the basic notions.

\subsection{Special flows}\label{se:potokispecjalne}

Let $T\colon (X,\mathcal{B},\mu)\to(X,\mathcal{B},\mu)$ be an ergodic automorphism of a standard probability space and let $f\in L^1 (X,\mathcal{B},\mu)$ be a strictly positive function. Let $X^f=\{(x,t)\in X\times\mathbb{R} \colon 0\leq t<f(x)\}$. Under the action of the \emph{special flow} $T^f$ each point of $X^f$ moves vertically upwards  at the unit speed and we identify the points $(x,0)$ and $(Tx,f(x))$. We put
\begin{displaymath}
f^{(m)}(x) = \left\{ \begin{array}{ll}
f(x)+f(Tx)+\ldots+f(T^{m-1}x) & \textrm{if $m>0$}\\
0 & \textrm{if $m=0$}\\
-(f(T^mx)+\ldots+f(T^{-1}x)) & \textrm{if $m<0$}.
\end{array} \right.
\end{displaymath}
For a formal definition of the special flow, consider the skew product $S_{-f}\colon (X\times \mathbb{R},\mu\otimes m)\to(X\times \mathbb{R},\mu\otimes m)$ given by
\begin{equation*}
S_{-f}(x,r)=(Tx,r-f(x))
\end{equation*}
and let $\Gamma^f$ stand for the quotient space $X\times\mathbb{R}/\sim$, where the relation $\sim$ identifies the points in each orbit of the action on $X\times\mathbb{R}$ by $S_{-f}$. Let $\sigma=(\sigma_t)_{t\in\mathbb{R}}$ denote the flow on $(X\times \mathbb{R},\mu\otimes m)$ given by
\begin{equation*}
\sigma_t(x,r)=(x,r+t).
\end{equation*}
Since $\sigma_t \circ S_{-f}=S_{-f}\circ \sigma_t$, we can consider the quotient flow of the action $\sigma$ by the relation $\sim$. This is the special flow over $T$ under $f$ denoted by $T^f$.

\subsection{Continued fractions}\label{se:cf}
For $\alpha\in (0,1)$ irrational we denote by $Tx=x+\alpha$ the corresponding irrational rotation on $\mathbb{T}$, where $\T$ is equipped with the $\sigma$-algebra of the Borel subsets and the Lebesgue measure inherited from $[0,1)$. Consider the continued fraction expansion of $\alpha$:
$$
\alpha = \cfrac{1}{a_1
          + \cfrac{1}{a_2
          + \cfrac{1}{\dots}}}.
$$
By $(q_n)_{n\geq 1}$ we denote the sequence of the denominators and by $(p_n)_{n\geq 1}$ the sequence of the numerators in the continued fraction expansion of $\alpha$:
\begin{align*}
p_0=0,\ p_1=1,\ p_{n+1}&=a_{n+1}p_n+p_{n-1},
\\
q_0=1,\ q_1=a_1,\ q_{n+1}&=a_{n+1}q_n+q_{n-1}.
\end{align*}

\section{Tools}\label{se:tools}
Lemańczyk and Parreau in~\cite{lemanczyk+parreau} proved a proposition which can be used for showing that some flows enjoy the SCS property. Before we state it, let us introduce the necessary notation. Denote by $\cP(\R)$ the space of probability Borel measures on $\R$ (endowed with the weak-$\ast$-topology). Let $CB(\R^n)$ stand for the space of continuous bounded functions on $\R^n$, let $C_n\colon \R^n\to \R$ be given by $C_n(x_1,\dots, x_n)=x_1+\dots+x_n$ and let $\mathcal{B}_{C_n}(\R^n)=C_n^{-1}(\mathcal{B}(\R))$, where $\mathcal{B}(\R)$ is the Borel $\sigma$-algebra on $\R$.

\begin{pr}\label{pr:3.1}
Let $\sigma\in\cP(\R)$ be continuous. Fix $n\geq 1$. Assume that $\mathcal{F}\subset CB(\R^n)$ (in particular $\mathcal{F}\subset L^2(\R^n,\sigma^{\otimes n})$) is a countable family $(\sigma^{\otimes n})$-a.e. measurable with respect to $\mathcal{B}_{C_n}(\R^n)$. Assume moreover that there exist $\widetilde{A}\subset \R^n$, $\sigma^{\otimes n}(\widetilde{A})=1$ and $B\subset \R$, $\sigma^{\ast n}(B)=1$ such that for each $c\in B$ if $(x_1,\dots, x_n), (x_1',\dots,x_n')\in C_n^{-1}(c)\cap \widetilde{A}$ and $J(x_1,\dots,x_n)=J(x_1',\dots,x_n')$ for each $J\in \mathcal{F}$ then $(x_1,\dots, x_n)=(x'_{\pi(1)},\dots,x'_{\pi(n)})$ for some permutation $\pi$ of $\{1,\dots,n\}$. Then for $\sigma^{\ast n}$-a.e. $c\in\R$ the conditional measure $\sigma_c^{(n)}$ is purely atomic concentrated on $n!$ atoms.
\end{pr}
In other words, given $\sigma\in\mathcal{P}(\mathbb{R})$ and $n\geq 1$, in order to show that $V_\sigma^{\odot n}$ has simple spectrum, one needs to develop methods for finding the countable family of functions which satisfies two conditions:
\begin{itemize}
\item[(i)]
the functions in this family are measurable with respect to the $\sigma$-algebra generated by the partition of $\mathbb{R}^n$ into the lines of the form $x_1+\dots+x_n=c$, $c\in\mathbb{R}$,
\item[(ii)]
the family distinguishes points from almost every line $x_1+\dots+x_n=c$.
\end{itemize}
\begin{uw}\label{uw:3-1}\cite{lemanczyk+parreau}
A countable family of tensors of the form $J=\hat{P}(\cdot)\otimes \dots\otimes \hat{P}(\cdot)$, where $P\in\cP(\R)$ is such that for some $t_n\to\infty$
\begin{equation}\label{eq:zbieznosc}
U_{T_{t_n}} \to \int U_{T_t}\ dP(t)\text{ in the weak operator topology}\footnote{The intergral in the right-hand side of the formula is defined weakly:
$$
\left\langle \int U_{T_t}\ dP(t) f ,g\right\rangle=\int \left\langle U_{T_t}f,g\right\rangle\ dP(t).
$$}
\end{equation}
satisfies the measurability condition denoted above as~(i), i.e. there exists a measurable function $F\colon \R\to \mathbb{C}$ such that
$$
\hat{P}(x_1)\cdot\ldots\cdot \hat{P}(x_n)=F(x_1+\dots+x_n)\text{ for }\sigma^{\otimes n}\text{-almost every }(x_1,\dots,x_n).
$$
\end{uw}
To find such measures $P$ satisfying condition~\eqref{eq:zbieznosc}, we will use the following result.
\begin{pr}\cite{FL04}\label{pr:l+p}
Let $T$ be an ergodic automorphism of a standard probability Borel space $(X,\mathcal{B},\mu)$ and let $(q_n)_{n\in\N}$ be a rigidity sequence of $T$. Suppose that $f\in L^2(X,\mathcal{B},\mu)$ is a positive function with $\int_X f\ d\mu=1$. Let $f_0=f-\int f\ d\mu$. Moreover, suppose that the sequence $\left(f_0^{(q_n)}\right)_{n\in\N}$ is bounded in $L^2\left(X,\mathcal{B},\mu\right)$, $\left(f_0^{(q_n)}\right)_\ast(\mu)$ converges weakly to $P$ and there exists $c>0$ such that $f^{(k)}(x)\geq ck$ for a.a. $x\in X$ and for all $k\in\N$ large enough. Then
\begin{equation}\label{eq:warunek}
U_{T^f_{q_n}}\to \int U_{T^f_{-t}}\ dP(t).
\end{equation}
\end{pr}
Therefore in order to apply Proposition~\ref{pr:3.1}, we will need to:
\begin{itemize}
\item
find the limit distribution of $f^{(q_n)}_0$,
\item
check that the family of tensors of the form $\hat{P}(\cdot)\otimes\dots\otimes \hat{P}(\cdot)$ satisfies condition (ii) introduced above, before Remark~\ref{uw:3-1}.
\end{itemize}
It is clear that the more measures $P$ satisfying condition~\eqref{eq:warunek} we obtain, the easier it should be to make sure that condition (ii) holds. In our case condition~\eqref{eq:warunek} will be satisfied not only by some measure $P$, but also by its ``integer rescalings'', i.e. by the measures $P_m=(M_m)_\ast(P)$ for $m\in\Z$, where
$$
M_t\colon\R \to \R  \text{ is given by }M_tx=tx \text{ for }t\in\R.
$$
Therefore the functions
$$
(t_1,\dots,t_n)\mapsto \hat{P}_m(t_1)\cdot\ldots\cdot\hat{P}_m(t_n)
$$
will be measurable with respect to $\sigma$-algebra $\mathcal{B}_{C_n}(\R^n)$. Since for $m\in\Z$ we have
$$
\hat{P}_t(m)=\hat{P}_m(t)
$$
and the product of Fourier transforms of measures is the Fourier transform of their convolution, the functions
\begin{equation}\label{eq:rodzina}
(t_1,\dots,t_n)\mapsto (P_{t_1}\ast\dots\ast P_{t_n})\hat{}(m)
\end{equation}
will be measurable $\sigma^{\otimes n}$-almost everywhere with respect to $\sigma$-algebra $\mathcal{B}_{C_n}(\R^n)$. Since the integer Fourier coefficients of a measure on $\R$ determine its image via the function $x\mapsto x\!\!\mod 1$, therefore the function
\begin{equation}\label{eq:mojafunction}
(t_1,\dots,t_n)\mapsto (\cdot \!\!\!\!\mod 1)_\ast(P_{t_1}\ast\dots\ast P_{t_n})
\end{equation}
will be also measurable $\sigma^{\otimes n}$-almost everywhere with respect to $\sigma$-algebra $\mathcal{B}_{C_n}(\R_n)$. We will use the density of measure $(\cdot \!\!\!\!\mod 1)_\ast(P_{t_1}\ast\dots\ast P_{t_n})$ to prove that the family of functions given by~\eqref{eq:rodzina} satisfies condition (ii).
\section{Smooth flows on surfaces}\label{se:potoki}
The flows which we consider in this paper are some smooth flows on closed orientable surfaces of genus at least $2$, having saddle-connections, i.e. some orbits beginning and ending in a saddle point. We will use their representation as special flows over an irrational rotation under the so-called symmetric logarithm roof function. More precisely, we consider a special flow over an irrational rotations on the circle $Tx=x+\alpha\ (\!\!\!\!\mod 1)$, $\alpha\in (0,1)\cap \R\setminus \Q$ under a roof function of the form $f+f_1+c\colon [0,1)\to \R$, where
$$
f(x)=-\ln(x)-\ln(1-x)-2,
$$
$f_1\colon\T\to \R$ is an absolutely continuous function with zero average and $c\in\R$ is such that $f+f_1+c>0$. We denote this flow by $\mathcal{T}=(T_t)_{t\in\R}$. Such flows were considered by Blokhin~\cite{MR0370656} who provided a construction on each closed orientable surface of genus at least $2$ yielding such special representation. In order to prove the results concerning the SCS property, we restrict this class by requiring that the rotation number $\alpha$ of the base transformation is a sufficiently well approximable irrational, precisely speaking
$$
\lim_{k\to\infty}q_{n_k}^3\|q_{n_k}\alpha\|=0
$$
for some subsequence $(q_{n_k})$ of the sequence of denominators $(q_n)$ of $\alpha$ in its continued fraction expansion.

\section{Results}\label{se:wyniki}

\subsection{New tools - the main proposition}\label{se:wyniki1}
We will describe now a general method of showing that for some measures $\sigma$ the unitary flow $V_\sigma^{\odot n}$ has simple spectrum. The tools which we will use were introduced in Section~\ref{se:tools}. Before we state the main proposition of this section, we need one more definition.

\begin{df}\label{df:szybkispadek}
We say that function $F\colon (-\infty,0)\cup(0,\infty)\to \R$ enjoys property $\mathcal{W}$, whenever there exist $A>1$, $t\geq 0$ and $r\in(-1,0)$ such that the following conditions hold:
\begin{itemize}
\item
$|F(x)|<Ae^{-\frac{|x|}{A}}$ for $|x|>t$\ (condition $\mathcal{W}_1$),
\item
$|F(x)|<A|x|^r$ for $0<|x|<t$\ (condition $\mathcal{W}_2$).
\end{itemize}
\end{df}

\begin{pr}\label{tw:nowametoda} 
Let  $\cT=(T_t)_{t\in\R}$ be a weakly mixing flow with the maximal spectral type $\sigma$. Let the measure $P\in \cP(\R)$ be absolutely continuous, such that for any $m\in\Z\setminus\{0\}$ there exists a sequence $t_n\to\infty$ such that
$$
U_{T_{t_n}}\to \int_{\R} U_{T_t}\ d(M_m)_{\ast}(P)(t)
$$
and there exists $a\in\R$ such that the density $h$ of the measure $P$ satisfies the following conditions:
\begin{itemize}
\item[(i)]
$h|_{(-\infty,a)}\equiv 0$ (or $h_{(a,\infty)}\equiv 0$),
\item[(ii)]
function $h|_{(a,\infty)}$ (or $h|_{(-\infty,a)}$) is smooth, i.e. belongs to the $C^{\infty}$ class,
\item[(iii)]
in the right-hand side (or left-hand side) neighbourhood of point $a$ the function $h$ is of the form $h(x)=(x-a)^{-1/2}\cdot\widetilde{h}(x)$, where $\widetilde{h}$ is analytic at point $a$, we have $h(x)=\sum_{n=0}^{\infty}a_n (x-a)^n$ for $x$ sufficiently close to point $a$, additionally requiring that $a_0\neq 0$,\footnote{It follows from this assumption that for every $k\geq 1$ and every $t_1,\dots,t_{2k}>0$ the function $h_{t_1}\ast \dots \ast h_{t_{2k}}$ is analytic at $(t_1+\dots+ t_{2k})\cdot a$, see Proposition~\ref{pr:dwie}, page~\pageref{pr:dwie}.}
\item[(iv)]
for any $n\in\N$ the function $x\mapsto (x-a)^n \cdot\frac{d^n}{dx^n}h(x-a)$ has property $\mathcal{W}$.
\end{itemize}
Let $d\geq 1$. Assume that on a set of full measure $\sigma^{\otimes d}$ the values of the function
\begin{equation}\label{eq:ff}
\sum_{k_1+\dots+k_d=n}\prod_{i=1}^{d}\frac{a_{k_i}\Gamma(k_i+1/2)}{t_i^{k_i}}
\end{equation}
of variables $t_1,\dots,t_d$ determine the values of these variables (up to the permutation of the variables) under the assumption that $t_1+\dots +t_d=c$ for $\sigma^{\ast d}$-a.a. $c\in\R$. Then the unitary flow $V_{\sigma}^{\odot d}$ has simple spectrum.
\end{pr}
To keep the structure of the remainder of this section clear, we include now an outline of the proof of the above proposition.

\begin{proof}[Outline of the proof]
To prove the above theorem, it suffices to show that for $(t_1,\dots,t_d)$ belonging to a set of full measure $\sigma^{\otimes d}$, for $\sigma^{\ast d}$-a.e. $c\in\R$ the measure $(\cdot \text{ mod }1)_{\ast}(P_{t_1}\ast \dots \ast P_{t_d})$ determines the set $\{t_1,\dots,t_d\}$ when $t_1+\dots + t_d=c$. To this end we show that for $\sigma^{\ast d}$-a.e. $c\in\R$ the measure  $(\cdot \text{ mod }1)_{\ast}(P_{t_1}\ast \dots \ast P_{t_d})$ determines the values of expressions of the form~\eqref{eq:ff}, whenever $t_1+\dots +t_d=c$. 

The main tool in the proof are analytic functions. We investigate the series expansion of the density $h_{t_1}\ast \dots \ast h_{t_d}$ of the measure $P_{t_1}\ast \dots \ast P_{t_d}$. The analysis is relatively easy when $t_1,\dots,t_d>0$. The main difficulty here lies in the fact that what we can use is not function $h_{t_1}\ast \dots \ast h_{t_d}$ itself, but its image via the map $x \mapsto x \text{ mod }1$. Using Proposition~\ref{lm:wlasnoscPdlasplotu} we conclude that the function $h_{t_1}\ast \dots \ast h_{t_d}$ enjoys a similar property like the one in assumption \emph{(iv)} of Proposition~\ref{tw:nowametoda}. It follows that the function $(\cdot \text{ mod } 1)_{\ast}(h_{t_1}\ast \dots \ast h_{t_d})$ carries the full information about the series expansion around $(t_1+\dots +t_d)\cdot a$ of the function $h_{t_1}\ast \dots \ast h_{t_d}$ (see Section~\ref{p343}). The coefficients in this expansion let us find the values of the expressions given by~\eqref{eq:ff}.

We will now describe the difference between the procedure described above and the procedure which we apply in the general case when the sign of the numbers $t_1,\dots,t_d$ is not known. Since we have assumed that $d\neq 1$ is odd, the cardinalities of the sets
\begin{equation*}
\left\{s \in \{t_1,\dots,t_d\}\colon s>0 \right\}\text{ and }\left\{s \in \{t_1,\dots,t_d\}\colon s<0 \right\}
\end{equation*}
are of different parity. We consider first the convolutions $h_{s_1}\ast \dots \ast h_{s_k}$, where $\{s_1,\dots,s_k\}\subset \{t_1,\dots,t_d\}$ are the maximal subsets such that the numbers $s_1,\dots,s_k$ have the same sign and then the convolution $h_{t_1}\ast \dots \ast h_{t_d}$. It turns out that also in this case the function $(\cdot \text{ mod }1)_{\ast }(h_{t_1}\ast \dots \ast h_{t_d})$ carries the full information about the values of the expressions given by~\eqref{eq:ff} (see Section~\ref{p343}). To prove that it is indeed true, we use the fact that by Proposition~\ref{lm:wlasnoscPdlasplotu}, Remark~\ref{ponim} and Proposition~\ref{lm:3.33}, all derivatives of  $h_{t_1}\ast \dots \ast h_{t_d}$ have analogous property to the one in the assumption~\emph{(iv)} of Proposition~\ref{tw:nowametoda}.

To end the proof, we use Proposition~\ref{pr:l+p}.
\end{proof}

\subsection{New tools - technical details}\label{se:details}
This section includes technical details concerning convolutions of functions from a certain class, their derivatives and the coefficients appearing in their series expansions. The proof of Proposition~\ref{tw:nowametoda} is included in Section~\ref{se:techni}.

\subsubsection{Convolutions and derivatives}\label{subse:5.2.1}
We will deal with functions which for some $\overline{x}\in\R$ vanish on one of the intervals $(-\infty,\overline{x})$ or $(\overline{x},\infty)$. In such a situation we will write $F\colon (\overline{x},\infty)\to \R$ or $F \colon (-\infty,\overline{x})\to \R$ respectively. According to this notation the convolution of functions $F_1\in L^1(x_1,\infty), F_2\in L^1(x_2,\infty)$ is given by the formula
$$
F_1\ast F_2 (x)=\int_\R F_1(y)F_2(x-y)\ dy=\int_{x_1}^{x+x_2} F_1(y) F_2(x-y)\ dy,
$$
whereas the convolution of $F_1\in L^1(x_1,\infty)$, $F_2\in L^1(-\infty, x_2)$ is given by
\begin{multline*} 
F_1\ast F_2 (x)=\int_\R F_1(y)F_2(x-y)\ dy=\\
=
\left\{
\begin{array}{rl} 
\int_{x-x_2}^\infty F_1(y)F_2(x-y)\ dy & \text{for } x\geq x_1+x_2\\
\int_{x_1}^\infty F_1(y)F_2(x-y)\ dy & \text{for } x<x_1+x_2.  
\end{array} 
\right.
\end{multline*}
We will treat the convolution of functions $F_1\in L^1(x_1,\infty)$ and $F_2\in L^1(-\infty,x_2)$ as a function whose domain is the set $(-\infty,x_1+x_2)\cup(x_1+x_2,\infty)$.
We will use the above formulas mainly for $x_1=x_2=0$. For $F_1\in L^1(0,\infty), F_2\in L^1(0,\infty)$ we have
$$F_1\ast F_2(x)=\int_0^x F_1(y)F_2(x-y)\ dy,$$
whereas for $F_1\in L^1(0,\infty), F_2\in L^1(-\infty,0)$
\begin{equation*} 
F_1\ast F_2 (x)=
\left\{
\begin{array}{rl} 
\int_{x}^\infty F_1(y)F_2(x-y)\ dy & \text{for } x\geq 0 \\
\int_{0}^\infty F_1(y)F_2(x-y)\ dy & \text{for } x<0.
\end{array} 
\right.
\end{equation*}

We will say that a ffunction $F\colon (0,\infty)\to \R$ is \emph{analytic at zero}, whenever it can be extended to a function which is analytic at zero, i.e. whenever there exists $\vep>0$ and an analytic function $\widetilde{F}\colon (-\varepsilon,\varepsilon)\to\R$ such that $\widetilde{F}|_{(0,\varepsilon)}=F|_{(0,\varepsilon)}$. We will use similar terminology for $F\colon (-\infty,0)\to\R$.

\paragraph{Derivatives of convolutions - part I}

\begin{lm}\label{lm:pochodne}
Let $a>1$ and let the functions $F_1,F_2\colon (0,\infty)\to\R$ be such that: 
\begin{itemize}
\item
$F_1\in L^1(0,\infty)$,
\item
for any $c>0$ the function $F_1$ is uniformly continuous on the interval $[c,\infty)$, 
\item
$F_2$ is differentiable, 
\item
for any $c>0$ the function $F_2'$ is uniformly continuous on the interval $[c,\infty)$.  
\end{itemize}
Then for $x>0$
\begin{equation*}
\frac{d}{dx}\int_{0}^{\frac{x}{a}}F_1(y)F_2(x-y)\ dy=\int_0^{\frac{x}{a}}F_1(y) \frac{d}{dx}F_2(x-y)\ dy +\frac{1}{a}F_1\left(\frac{x}{a}\right)F_2\left(x-\frac{x}{a}\right).
\end{equation*}
\end{lm}

\begin{proof}
Fix $x>0$ and $a>1$ and notice that the integral $\int_{0}^{\frac{x}{a}}F_1(y)F_2(x-y)\ dy$ is finite. Indeed, by the assumption the function $F_2'$ is uniformly continuous on the interval $[x-\frac{x}{a},x]$, hence $F_2$ is bounded in this interval. Therefore and by the integrability of $F_1$ it follows that the considered integral is indeed finite.

Whenever $h>0$ is sufficiently small then for some $\theta=\theta(h)\in (0,h)$, using the triangle inequality and the mean value theorem we obtain
\begin{align*}
W_x(h):=
&\left| \frac{1}{h}\left(\int_{0}^{\frac{x+h}{a}}F_1(y)F_2(x+h-y)\ dy- \int_{0}^{\frac{x}{a}}F_1(y)F_2(x-y)\ dy \right)\right.\\
&\left.-\int_{0}^{\frac{x}{a}}F_1(y) F_2'(x-y)\ dy-\frac{1}{a}F_1\left(\frac{x}{a} \right)F_2\left(x-\frac{x}{a} \right) \right|\\
\leq& \int_{0}^{\frac{x}{a}}|F_1(y)|\left|\frac{F_2(x+h-y)-F_2(x-y)}{h}-F_2'(x-y)\right|\ dy\\
&+\int_{\frac{x}{a}}^{\frac{x+h}{a}}\frac{1}{h}\left|F_1(y)F_2(x+h-y)-F_1\left(\frac{x}{a}\right)F_2\left(x-\frac{x}{a} \right)\right|\ dy\\
=&\int_{0}^{\frac{x}{a}}|F_1(y)|\left|F_2'(x-y+\theta)-F_2'(x-y) \right|\ dy\\
&+\int_{\frac{x}{a}}^{\frac{x+h}{a}}\frac{1}{h}\left|F_1(y)F_2(x+h-y)-F_1\left(\frac{x}{a}\right)F_2\left(x-\frac{x}{a} \right)\right|\ dy.
\end{align*}
We may assume that $h>0$ is small enough, so that $x-\frac{x}{a}-h>0$. By the uniform continuity of the function $F_2'$ on the interval $\left[x-\frac{x}{a}-h,x \right]$, for $y\in \left[ 0,\frac{x}{a}\right]$ we obtain
\begin{equation*}
\left|F_2'(x-y+\theta)-F_2'(x-y) \right|<\varepsilon.
\end{equation*}
Therefore by the integrability of the function $F_1$
\begin{equation*}
\int_{0}^{\frac{x}{a}}|F_1(y)|\left|F_2'(x-y+\theta)-F_2'(x-y) \right|\ dy<\varepsilon
\end{equation*}
for $h>0$ small enough. By uniform continuity of the functions $F_1$ and $F_2$ on the intervals of the from $[c,\infty)$ for any $c>0$, for $y\in \left[\frac{x}{a},\frac{x+h}{a}\right]$ and $h>0$ small enough we obtain
\begin{equation*}
\left|F_1(y)F_2(x+h-y)-F_1\left(\frac{x}{a}\right)F_2\left(x-\frac{x}{a} \right)\right|<\varepsilon.
\end{equation*}
Hence
\begin{equation*}
\int_{\frac{x}{a}}^{\frac{x+h}{a}}\frac{1}{h}\left|F_1(y)F_2(x+h-y)-F_1\left(\frac{x}{a}\right)F_2\left(x-\frac{x}{a} \right)\right|\ dy <\varepsilon.
\end{equation*}
Therefore $\lim_{h\to 0^+}W_x(h)=0$. We treat the case where $h<0$ in a similar way and obtain $\lim_{h\to 0^-}W_x(h)=0$,
which ends the proof. 
\end{proof}

\begin{uw}\label{sto}
If the functions $F_1,F_2\colon (0,\infty)\to\R$ fulfill the assumptions of Lemma~\ref{lm:pochodne} and $z_0>0$, then the function
\begin{equation*}
x\mapsto \int_0^{x-z_0}F_1(y)F_2(x-y)\ dy
\end{equation*}
is differentiable on the interval $(z_0,\infty)$ and for $x>z_0$ we obtain
\begin{multline*}
\frac{d}{dx}\int_0^{x-z_0}F_1(y)F_2(x-y)\ dy\\
=\int_0^{x-z_0} F_1(y)\frac{d}{dx}F_2(x-y)\ dy+ F_1(x-z_0)F_2(z_0).
\end{multline*}
In the same way, the function
\begin{equation*}
z\mapsto \int_0^{z_0} F_1(y)F_2(x-y)\ dy
\end{equation*}
is differentiable on the interval $(z_0,\infty)$ and for $x>z_0$ we have
\begin{equation*}
\frac{d}{dx}\int_0^{z_0}F_1(y)F_2(x-y)\ dy= \int_0^{z_0}F_1(y)\frac{d}{dx}F_2(x-y)\ dy.
\end{equation*}
\end{uw}

\begin{lm}\label{wn:duzopochodnych}
Let $a>1$, $k\geq 1$ and let the functions $F_1,F_2\colon(0,\infty) \to \R$ be such that:
\begin{itemize}
\item
$F_1 \in L^1(0,\infty)$, 
\item
$F_1$ is differentiable $k-1$ times,
\item
for $0\leq l\leq k-1$ and for any $c>0$ the function $\frac{d^l}{dx^l}F_1$ is uniformly continuous on the interval $[c,\infty)$,
\item
$F_2$ is differentiable $k$ times,
\item
for $0\leq l\leq k$ and for any $c>0$ the function $\frac{d^l}{dx^l}F_2$ is uniformly continuous on the interval $[c,\infty)$.
\end{itemize}
Then
\begin{multline*}
\frac{d^k}{dx^k}\int_{0}^{\frac{x}{a}}F_1(y)F_2(x-y)\ dy\\
=\int_{0}^{\frac{x}{a}}F_1(y)\frac{d^k}{dx^k}F_2(x-y)\ dy+\sum_{l=0}^{k-1}w_l \frac{d^l}{dx^l}F_1\left(\frac{x}{a} \right)\frac{d^{k-l-1}}{dx^{k-l-1}}F_2\left( x-\frac{x}{a}\right)
\end{multline*}
for some $w_l\in \R$, $0\leq l\leq k-1$, depending on $k$ and $a$.
\end{lm}
\begin{proof}
It suffices to use Lemma~\ref{lm:pochodne} and the mathematical induction.
\end{proof}

\begin{uw}\label{sto1}
Under assumptions as in Lemma~\ref{wn:duzopochodnych} one can show that given $z_0>0$ the function
\begin{equation*}
x\mapsto \int_0^{x-z_0}F_1(y)F_2(x-y)\ dy
\end{equation*}
is $k$-times differentiable in the interval $(z_0,\infty)$ and for $x>z_0$ the following equality holds:
\begin{multline*}
\frac{d^k}{dx^k}\int_0^{x-z_0}F_1(y)F_2(x-y)\ dy \\
= \int_0^{x-z_0}F_1(y)\frac{d^k}{dx^k} F_2(x-y)\ dy+\sum_{l=0}^{k-1}\frac{d^l}{dx^l}F_1(x-z_0)\frac{d^{k-l-1}}{dx^{k-l-1}}F_2(z_0).
\end{multline*}
Moreover, the function
\begin{equation*}
x\mapsto \int_0^{z_0}F_1(y)F_2(x-y)\ dy
\end{equation*}
is $k$-time differentiable in the interval $(z_0,\infty)$ and for $x>z_0$ the following equality holds:
\begin{equation*}
\frac{d^k}{dx^k}\int_{0}^{z_0}F_1(y)F_2(x-y)\ dy=\int_{0}^{z_0}F_1(y)\frac{d^k}{dx^k}F_2(x-y)\ dy.
\end{equation*}
\end{uw}

\begin{pr}\label{uw:uw2}
Let $k\geq 1$ and let the functions $F_1,F_2\colon (0,\infty)\to \R$ be such that:
\begin{itemize}
\item
$F_1,F_2\in L^1(0,\infty)$,
\item
for any $c>0$ the functions $\frac{d^l}{dx^l}F_i$ are uniformly continuous on the interval $[c,\infty)$ for $i=1,2$ for $0\leq l\leq k$.
\end{itemize}
Then for $x>0$
\begin{multline}\label{eq:wzorpoch}
\frac{d^k}{dx^k}(F_1\ast F_2)(x)=\int_0^{\frac{x}{2}} F_1(y)\frac{d^k}{dx^k} F_2(x-y)\ dy\\
+\int_0^{\frac{x}{2}} F_2(y)\frac{d^k}{dx^k} F_1(x-y)\ dy+\sum_{l=0}^{k-1} \overline{w}_l \frac{d^l}{dx^l}F_1\left(\frac{x}{2} \right) \frac{d^{k-1-l}}{dx^{k-1-l}}F_2\left(\frac{x}{2} \right)
\end{multline}
for some $\overline{w}_l\in\R$, $0\leq l\leq k-1$. Moreover, for $z_0>0$ and $x\in (z_0,\infty)$ we have the following formula
\begin{multline}\label{eq:wzorpoch2}
\frac{d^k}{dx^k}(F_1\ast F_2)(x)=\int_0^{z_0} F_1(y)\frac{d^k}{dx^k} F_2(x-y)\ dy\\
+\int_0^{x-z_0} F_2(y)\frac{d^k}{dx^k} F_1(x-y)\ dy\\
+\sum_{l=0}^{k-1} \frac{d^l}{dx^l}F_1\left(z_0\right) \frac{d^{k-1-l}}{dx^{k-1-l}}F_2\left(x-z_0 \right).
\end{multline}
\end{pr}
\begin{proof}
It suffices to apply Lemma~\ref{wn:duzopochodnych} and Remark~\ref{sto1}.
\end{proof}

\begin{lm}
Let $F_1,F_2\in L^1(0,\infty)$. If
\begin{equation*}
\lim_{x\to \infty}F_1(x)=\lim_{x\to\infty}F_2(x)=0,
\end{equation*}
then $\lim_{x\to\infty} (F_1\ast F_2)(x)=0$. Moreover, if the functions $F_1$ and $F_2$ are $k$ times differentiable for some $k\geq 1$ and
\begin{equation*}
\lim_{x\to \infty}\frac{d^l}{dx^l}F_1(x)=\lim_{x\to\infty}\frac{d^l}{dx^l}F_2(x)=0
\end{equation*}
for $0\leq l\leq k$, then
\begin{equation*}
\lim_{x\to \infty}\frac{d^k}{dx^k} F_1 \ast F_2(x) =0.
\end{equation*}
\end{lm}
\begin{proof}
If suffices to use formula~\eqref{eq:wzorpoch} from Proposition~\ref{uw:uw2}.
\end{proof}

\paragraph{Derivatives of convolutions - part II}

We will now deal with the problem of calculating the derivative of a convolution of two functions which vanish of two complementary halflines.
\begin{pr}\label{stw:deriv}
Let $k\geq 1$, $F_1\in L^1(0,\infty)$, $F_2\in L^1(-\infty,0)$. Assume that:
\begin{itemize}
\item
the functions $F_1$ and $F_2$ are $k$ times differentiable, 
\item
for all $c>0$ and $0\leq l\leq k$ the function $\frac{d^l}{dx^l}F_1$ is uniformly continuous on the interval $[c,\infty)$. 
\item
the function $F_2$ is analytic at zero, 
\item
for all $c>0$ and $0\leq l\leq k$ the function $\frac{d^l}{dx^l}F_2$ is uniformly continuous on the interval $(-\infty,-c]$.\footnote{Under these assumptions all the derivatives of the function $F_2$, as functions which are analytic at zero and uniformly continuous on the interval $(-\infty,-c]$ for any $c>0$, are uniformly continuous on the interval $(-\infty,0]$.\label{bbb}}
\end{itemize}
Then the function $F_1\ast F_2$ is also $k$ times differentiable in the set $(-\infty,0)\cup(0,\infty)$ and the following formulas hold:
\begin{equation*}
\frac{d^k}{dx^k} F_1\ast F_2 (x)=\int_x^{\infty}F_1(y)\frac{d^k}{dx^k}F_2(x-y)\ dy - \sum_{l=0}^{k-1}\frac{d^{k-l-1}}{dx^{k-l-1}}F_1(x) \frac{d^l}{dx^l}F_2(0)
\end{equation*}
for $x>0$ and
\begin{equation*}
\frac{d^k}{dx^k} F_1\ast F_2 (x)=\int_0^{\infty}F_1(y)\frac{d^k}{dx^k}F_2(x-y)\ dy.
\end{equation*}
for $x<0$.
\end{pr}

\begin{proof}
We will show first that for $x>0$ we have
\begin{equation}\label{eq:toto}
\frac{d}{dx}F_1\ast F_2(x)=\int_x^{\infty}F_1(y)F_2'(x-y)\ dy-F_1(x)F_2(0).
\end{equation}
Fix $x>0$. For $h>0$ sufficiently small we have
\begin{multline}\label{eq:zerro}
\left|  \frac{F_1\ast F_2(x+h)-F_1\ast F_2 (x)}{h}-\int_x^{\infty}F_1(y)F_2'(x-y)\ dy+F_1(x)F_2(0)\right|\\
= \left| \frac{1}{h}\int_{x+h}^{\infty}F_1(y)F_2(x+h-y)\ dy -\frac{1}{h}\int_x^{\infty}F_1(y)F_2(x-y)\ dy\right.\\
\left.-\int_x^{\infty}F_1(y)F_2'(x-y)\ dy + F_1(x)F_2(0)\right|\\
\leq\int_{x+h}^{\infty}|F_1(y)|\left|\frac{F_2(x+h-y)-F_2(x-y)}{h}-F_2'(x-y) \right|\ dy\\
+\left|F_1(x)F_2(0)-\frac{1}{h}\int_x^{x+h}F_1(y)F_2(x+h-y)\ dy \right|\\
+\left|\int_x^{x+h}F_1(y)F_2'(x-y)\ dy \right|.
\end{multline}
There exists $\theta=\theta(y,h) \in (0,h)$ such that
\begin{equation*}
\frac{F_2(x+h-y)-F_2(x-y)}{h}=F_2'(x-y+\theta).
\end{equation*}
By uniform continuity of the function $F_2'$ on the interval $(-\infty,0]$ (see the remark in the footnote on page~\pageref{bbb}), for $h>0$ sufficiently small and for $y\in [x+h,\infty)$ the following inequality holds: $|F_2'(x-y+\theta)-F_2'(x-y)|<\varepsilon$. Therefore and by the integrability of $F_1$ we obtain
\begin{equation}\label{eq:adin}
\int_{x+h}^{\infty}|F_1(y)|\left|\frac{F_2(x+h-y)-F_2(x-y)}{h}-F_2'(x-y) \right|\ dy<\varepsilon
\end{equation} 
for small enough $h>0$. Moreover,
\begin{multline*}
\left|F_1(x)F_2(0)-\frac{1}{h}\int_x^{x+h}F_1(y)F_2(x+h-y)\ dy \right|\\
\leq\frac{1}{h}\int_x^{x+h} \left|F_1(y)F_2(x+h-y)-F_1(x)F_2(0) \right|\ dy.
\end{multline*}
Since the function $F_1$ is continuous at point $x$, the function $F_2$ is uniformly continuous, therefore for small $h>0$ we obtain
\begin{equation*}
|F_1(y)F_2(x+h-y)-F_1(x)F_2(0)|<\varepsilon
\end{equation*}
for any $y\in(x,x+h)$. Therefore
\begin{equation}\label{eq:dwwa}
\frac{1}{h}\int_x^{x+h} \left|F_1(y)F_2(x+h-y)-F_1(x)F_2(0) \right|\ dy<\varepsilon.
\end{equation}
Since the function $F_1$ is integrable, and the function $F_2'$, being is uniformly continuous, is bounded on the interval $(x,x+h)$, therefore
\begin{equation}\label{eq:trri}
\left|\int_x^{x+h}F_1(y)F_2'(x-y)\ dy \right|<\varepsilon
\end{equation}
when $h>0$ is sufficiently small. By~\eqref{eq:zerro},~\eqref{eq:adin},~\eqref{eq:dwwa} and~\eqref{eq:trri} it follows that
\begin{equation*}
(F_1\ast F_2)'_+(x)=\int_x^{\infty}F_1(y)F_2'(x-y)\ dy-F_1(x)F_2(0).
\end{equation*}
In a similar way one can show that
\begin{equation*}
(F_1\ast F_2)'_-(x)=\int_x^{\infty}F_1(y)F_2'(x-y)\ dy-F_1(x)F_2(0),
\end{equation*}
whence the formula~\eqref{eq:toto} indeed holds. Using the same methods one can show that the formula for the first derivative also holds for negative arguments. The formulas for the higher order derivatives hold by the induction.
\end{proof}

\subsubsection{Series expansions and rescaled densities}\label{sub:rws}

\paragraph{Series expansions for convolutions}
We will now look for properties of convolutions and their derivatives for functions $F_i\in L^1(0,\infty)$, $i\in \N$, which in the right-hand neighbourhood of zero are of the form
$$
F_i(x)=x^{a_i} \cdot \widetilde{F}_i(x),
$$
where $a_i>-1$ and $\widetilde{F}_i\colon (0,\infty)\to\R$ is a function which is analytic at zero (i.e. for $x>0$ small enough we have $\widetilde{F}_i(x)=\sum_{n=0}^{\infty}a_{i,n} x^n$).\footnote{Note that condition $a_i>-1$ is necessary for $F_i$ to be integrable.}

\begin{uw}\label{lm:calki}
Recall that beta function for $x,y\in\mathbb{C}$ with $\text{Re}(x),\text{Re}(y)>0$ is defined by $B(x,y)=\int_0^1 t^{x-1}(1-t)^{y-1}\ dt$ and $B(x,y)=\frac{\Gamma(x)\Gamma(y)}{\Gamma(x+y)}$ (see e.g.~\cite{abramowitz+stegun}). Hence for  $a_1,a_2>-1$ we have
\begin{multline*}
\int_0^x y^{a_1}\cdot (x-y)^{a_2}\ dy\\
=x^{a_1+a_2+1}\int_0^1 y^{a_1}(1-y)^{a_2}\ dy=\frac{\Gamma(a_1+1)\Gamma(a_2+1)}{\Gamma(a_1+a_2+2)}x^{a_1+a_2+1}.
\end{multline*}

\end{uw}

\begin{pr}\label{pr:dwie}
Assume that the functions $F_1,F_2\in L^1(0,\infty)$ satisfy the following conditions:
\begin{itemize}
\item
in the right-hand side neighbourhood of zero $F_1(x)=x^{a_1} \cdot \widetilde{F}_1(x)$, where $a_1\in\R$, the function $\widetilde{F}_1$ is analytic at zero with the series expansion around zero given by $\widetilde{F}_1(x)=\sum_{n=0}^{\infty}B_nx^n$,
\item
in the right-hand side neighbourhood of zero $F_2(x)=x^{a_2} \cdot \widetilde{F}_2(x)$, where $a_2\in\R$, the function $\widetilde{F}_2$ is analytic at zero with the series expansion around zero given by  $\widetilde{F}_2(x)=\sum_{n=0}^{\infty}C_nx^n$.
\end{itemize}
Then the function $F_1\ast F_2$ is of the following form in the right-hand side neighbourhood of zero:
\begin{multline}\label{eq:wspdla2}
F_1\ast F_2(x)\\
=\sum_{n=0}^{\infty} \left(\sum_{k=0}^{n}\left(B_k C_{n-k}\frac{\Gamma(k+a_1+1)\Gamma(n-k+a_2+1)}{\Gamma(n+a_1+a_2+2)} \right) x^{n+a_1+a_2+1}\right).
\end{multline}
\end{pr}

\begin{proof}
For $n\in\N$ let 
$$f_n(x,y)=\sum_{k=0}^{n}B_kC_{n-k}y^k(x-y)^{n-k}.$$
By analicity of the functions $\widetilde{F}_1, \widetilde{F}_2$ it follows that there exists $x_0>0$ such that for $x\in(0,x_0)$ we have
\begin{multline*}
F_1\ast F_2 (x)= \int_{0}^{x} y^{a_1}\cdot (x-y)^{a_2}\cdot\left(\sum_{n=0}^{\infty}B_ny^n\right) \cdot \left(\sum_{n=0}^{\infty}C_n (x-y)^n \right)\ dy\\
=\int_0^x \left( y^a_1\cdot (x-y)^{a_2}\cdot \sum_{n=0}^{\infty}f_n(x,y)\right)\ dy\\
= \int_{0}^{x}\left( y^a_1\cdot (x-y)^{a_2}\cdot\sum_{n=N+1}^{\infty}f_n(x,y)\right)\ dy+\sum_{n=0}^{N}\int_{0}^{x}y^a_1\cdot (x-y)^{a_2}\cdot f_n(x,y)\ dy.
\end{multline*}

Notice that for $N\in\N$ sufficiently large $\left|\sum_{n=N+1}^{\infty}f_n(x,y) \right|<\varepsilon$ for any $y\in(0,x)$. Indeed, we may assume that the series $\sum_{n=0}^{\infty}|B_n|\cdot x_0^n$ and $\sum_{n=0}^{\infty}|C_n|\cdot x_0^n$ converge. Since for $y\in (0,x)\subset (0,x_0)$ we have
\begin{equation*}
|B_k\cdot y^k| \leq |B_k| \cdot x_0^k \text{ and }|C_{n-k}\cdot (x-y)^{n-k}|\leq |C_{n-k}\cdot x_0^{n-k}|,
\end{equation*}
therefore by the Mertens' theorem on convergence of Cauchy products of series and by the comparison test the series $\sum_{n=0}^{\infty}f_n(x,y)$ converges uniformly in $x\in (0,x_0)$ and $y\in (0,x)$.

Therefore by Remark~\ref{lm:calki}
\begin{multline*}
\left|\int_0^x y^a_1\cdot (x-y)^{a_2}\cdot\sum_{n=N+1}^{\infty} f_n(x,y)\ dy\right|\\
<\varepsilon \cdot \frac{\Gamma(a_1+1)\Gamma(a_2+1)}{\Gamma(a_1+a_2+2)}x^{a_1+a_2+1}
\leq \varepsilon \cdot \frac{\Gamma(a_1+1)\Gamma(a_2+1)}{\Gamma(a_1+a_2+2)}x_0^{a_1+a_2+1}
\end{multline*}
for $N$ sufficiently large. Therefore
\begin{equation*}
\int_0^x \sum_{n=0}^{\infty}y^a_1\cdot (x-y)^{a_2}\cdot f_n(x,y)\ dy = \sum_{n=0}^{\infty} \int_{0}^{x} y^a_1\cdot (x-y)^{a_2}\cdot f_n(x,y)\ dy.
\end{equation*}
Using again Remark~\ref{lm:calki}, we obtain
\begin{multline*}
\int_0^x y^a_1\cdot (x-y)^{a_2}\cdot f_n(x,y)\ dy\\
=\left(\sum_{k=0}^n B_k C_{n-k}\frac{\Gamma(k+a_1+1)\Gamma(n-k+a_2+1)}{\Gamma(n+a_1+a_2+1)}\right)x^{n+a_1+a_2+1},
\end{multline*}
which ends the proof.
\end{proof}

\paragraph{Convolutions of rescaled functions and their series expansions}
For a function $F\colon \R\to \R$ and $t\neq 0$ we will use the following notation:
$$
F_t(x)=\frac{1}{|t|}F\left(\frac{x}{t}\right),\ F_t\colon \R \to \R.
$$

\begin{wn}\label{wn:9}
Assume that the function $F\in L^1(0,\infty)$ in the right-hand side neighbourhood of zero is of the form $F(x)=x^a\cdot \widetilde{F}(x)$, where $a>-1$, $\widetilde{F}$ is  analytic at zero with the expansion around zero given by $\widetilde{F}(x)=\sum_{n=0}^{\infty}a_nx^n$. Then for $d\geq 1$ and $t_1,\dots,t_d>0$ in the right-hand side neighbourhood of zero we have
\begin{multline}\label{eq:rdrd}
F_{t_1}\ast \dots \ast F_{t_d}(x)=\frac{x^{d\cdot a}}{(t_1\cdot\ldots\cdot t_d)^{1+a}}\cdot\\
\cdot \sum_{n=0}^{\infty}\left(\sum_{k_1+\dots+ k_d=n}\prod_{i=1}^{d} \frac{a_{k_i}\cdot \Gamma(k_i+1+a)}{t_i^{k_i}}\cdot\frac{1}{\Gamma(n+d+d\cdot a)}\right)x^{n+d-1}.
\end{multline}
In particular, for $d\geq 1$ and $t_1,\dots,t_d>0$ we have:
\begin{itemize}
\item
for $a=0$
\begin{multline*}
F_{t_1}\ast \dots \ast F_{t_d}(x)=\frac{1}{t_1\cdot\ldots\cdot t_d}\cdot\\
\cdot \sum_{n=0}^{\infty}\left(\sum_{k_1+\dots+ k_d=n}\prod_{i=1}^{d} \frac{a_{k_i}\cdot \Gamma(k_i+1)}{t_i^{k_i}}\cdot\frac{1}{\Gamma(n+d)}\right)x^{n+d-1},
\end{multline*}
\item
for $a=-\frac{1}{2}$ and $d$ even
\begin{multline*}
F_{t_1}\ast \dots \ast F_{t_d}(x)=\frac{x^{-d/2}}{(t_1\cdot\ldots\cdot t_d)^{1/2}}\cdot\\
\cdot \sum_{n=0}^{\infty}\left(\sum_{k_1+\dots+ k_d=n}\prod_{i=1}^{d} \frac{a_{k_i}\cdot \Gamma(k_i+1/2)}{t_i^{k_i}}\cdot\frac{1}{\Gamma(n+d/2)}\right)x^{n+d-1}.
\end{multline*}
\item
for $a=-\frac{1}{2}$ and $d$ odd
\begin{multline}\label{eq:wzorparz}
F_{t_1}\ast \dots \ast F_{t_d}(x)=\frac{x^{\frac{-d+1}{2}}}{(t_1\cdot\ldots\cdot t_d\cdot x)^{1/2}}\cdot\\
\cdot \sum_{n=0}^{\infty}\left(\sum_{k_1+\dots+ k_d=n}\prod_{i=1}^{d} \frac{a_{k_i}\cdot \Gamma(k_i+1/2)}{t_i^{k_i}}\cdot\frac{1}{\Gamma(n+d/2)}\right)x^{n+d-1}.
\end{multline}
\end{itemize}
Moreover, for $a=0$, $d\geq 1$ and $t_1,\dots,t_d<0$ we have
\begin{multline*}
F_{t_1}\ast \dots \ast F_{t_d}(x)=-\frac{1}{t_1\cdot\ldots\cdot t_d}\cdot\\
\cdot \sum_{n=0}^{\infty}\left(\sum_{k_1+\dots+ k_d=n}\prod_{i=1}^{d} \frac{a_{k_i}\cdot \Gamma(k_i+1)}{t_i^{k_i}}\cdot\frac{1}{\Gamma(n+d)}\right)x^{n+d-1},
\end{multline*}
whereas for $a=-\frac{1}{2}$ and $t_1,\dots,t_d<0$ we have
\begin{itemize}
\item 
for $d$ even
\begin{multline}\label{eq:wzorparz2}
F_{t_1}\ast \dots \ast F_{t_d}(x)=(-1)^{\frac{d-2}{2}}\frac{x^{-d/2}}{(t_1\cdot\ldots\cdot t_d)^{1/2}}\cdot\\
\cdot \sum_{n=0}^{\infty}\left(\sum_{k_1+\dots+ k_d=n}\prod_{i=1}^{d} \frac{a_{k_i}\cdot \Gamma(k_i+1/2)}{t_i^{k_i}}\cdot\frac{1}{\Gamma(n+d/2)}\right)x^{n+d-1},
\end{multline}
\item
for $d$ odd
\begin{multline*}
F_{t_1}\ast \dots \ast F_{t_d}(x)=(-1)^{\frac{-d+1}{2}}\frac{x^{\frac{-d+1}{2}}}{(t_1\cdot\ldots\cdot t_d\cdot x)^{1/2}}\cdot\\
\cdot \sum_{n=0}^{\infty}\left(\sum_{k_1+\dots+ k_d=n}\prod_{i=1}^{d} \frac{a_{k_i}\cdot \Gamma(k_i+1/2)}{t_i^{k_i}}\cdot\frac{1}{\Gamma(n+d/2)}\right)x^{n+d-1}.
\end{multline*}
\end{itemize}
\end{wn}

\begin{proof}
To prove the first of the announced formulas it suffices to use Proposition~\ref{pr:dwie} and the induction on the number of functions. The formulas for $t_1,\dots,t_d<0$ 
follow from the formulas for $t_1,\dots,t_d>0$ and the relations
$$
F_{t_1}\ast \dots \ast F_{t_d}(x)=F_{-t_1}\ast \dots\ast F_{-t_d}(-x)
$$
for $x\in\R$.
\end{proof}

\begin{uw}
Under the assumptions of Corollary~\ref{wn:9} (and using the same notation), for $a=0$ and $t_1,\dots,t_d$ of the same sign the function $F_{t_1}\ast\dots\ast F_{t_d}$ is analytic at zero, and for $a=-\frac{1}{2}$ and $t_1,\dots,t_d$ of the same sign it is analytic at zero for $d$ even.
\end{uw}

\subsubsection{Finding coefficients of series expansions}\label{p343}

\paragraph{General case}

\begin{df}\label{def:dostszyb}
Let $F$ be a smooth (i.e. of class $C^{\infty}$) real-valued function with the domain $(0,\infty)$, $(-\infty,0)$ or $(-\infty,0)\cup (0,\infty)$. We say that \emph{all derivatives of $F$ decay ``sufficiently fast''}, if for any $n\geq 0$ there exist numbers $x_n>0$ and non-increasing functions $H_n\in L^1(x_n,\infty)$ such that
\begin{equation*}
\left|\frac{d^n}{dx^n}F(x) \right|\leq H_n(|x|) \text{ for }|x|>x_n.
\end{equation*}
\end{df}

\begin{lm}\label{lm:pomo}
Let $F \colon (-\infty,0)\cup(0,\infty) \to \R$ be a function such that all its derivatives decay ``sufficiently fast''\footnote{The function $F$ may vanish on one of the half-lines $(-\infty,0)$ or $(0,\infty)$.}. Then the function $G\colon \left(-\frac{1}{2},\frac{1}{2}\right) \to \R$ given by the formula $G(x)= \sum_{k\neq 0,k\in \Z}F(x)$ is smooth. In particular
\begin{equation}\label{eq:conti}
\lim_{x\to 0}\left(\frac{d^n}{dx^n}G(x)-\frac{d^n}{dx^n}G(-x)\right)=0
\end{equation}
for $n\geq 0$.
\end{lm}

\begin{proof}
By the assumption we have
\begin{equation*}
\sum_{|k|\geq x_0+1,k\in\Z}\left|F(x+k) \right|\leq \sum_{|k|\geq x_0+1,k\in\Z}H_0(|x+k|)
\end{equation*}
for $x\in(-\frac{1}{2},\frac{1}{2})$.
Notice that by monotonicity of the function $H_0$ and the comparison test, the series $\sum_{k\neq 0,k\in\Z}F(x+k)$ is uniformly convergent on the interval $\left(-\frac{1}{2},\frac{1}{2}\right)$. Therefore $G$ is continuous. In the same way one can show that $G$ is smooth. Indeed, we have
\begin{equation*}
\sum_{|k|\geq x_n+1,k\in\Z}\left|\frac{d^n}{dx^n}F(x+k) \right|\leq \sum_{|k|\geq x_n+1,k\in\Z}H_n(|x+k|),
\end{equation*}
where the series on the right-hand side of the inequality is uniformly convergent, whence we can interchange the order of differentiation and summing the series in the calculations to obtain the derivatives of $G$:
\begin{equation*}
\frac{d^n}{dx^n}G(x)=\frac{d^n}{dx^n}\left(\sum_{k\neq 0,k\in\Z}F(x+k) \right)=\sum_{k\neq 0,k\in\Z}\frac{d^n}{dx^n}F(x+k).
\end{equation*}
Therefore $G\in C^{\infty}\left(-\frac{1}{2},\frac{1}{2} \right)$.
\end{proof}

\begin{pr}\label{pr:nawijanie1}
Assume that the function $F\in L^1(0,\infty)$ satisfies the following conditions:
\begin{itemize}
\item
in the right-hand side neighbourhood of zero $F$ it is of the form $F(x)=x^a\cdot \widetilde{F}(x)$, where $a>-1$, the function $\widetilde{F}\colon (0,\infty)\to \R$ is analytic at zero with the series expansion in the right-hand side neighbourhood of zero given by $\widetilde{F}(x)=\sum_{n=0}^{\infty}A_nx^n$,
\item
$F$ is smooth,
\item
all the derivatives of $F$ decay ``sufficiently fast''.
\end{itemize}
Then the function $(\cdot \text{ mod }1)_{\ast}(F)$ determines the coefficients $A_n$ for $n\in \N$.\footnote{Recall that the function $x \mapsto x\text{ mod }1$ assigns the fractional part to reals, whence for any function $F$ the following formula holds: $(\cdot \text{ mod }1)_{\ast}(F)(x)=\sum_{k\in \Z}F(x+k)$.}
\end{pr}

\begin{proof}
Since $a=\left([a]+1\right)+\left(\{a\}-1\right)$, we have
$$
x^a\cdot \sum_{n=0}^{\infty}A_nx^n=x^{\{a\}-1}\cdot \sum_{n=0}^{\infty} A_n x^{n+[a]+1}.
$$
Therefore without loss of generality we may assume that $a\leq 0$.

Let the function $G\colon \left(-\frac{1}{2},\frac{1}{2}\right)\to\R$ be given by the formula
\begin{equation*}
G(x):=\sum_{k\neq 0,k\in\Z}F(x+k)=\sum_{k=1}^{\infty}F(x+k).
\end{equation*}
By Lemma~\ref{lm:pomo}
\begin{equation*}
\lim_{x\to 0}\left(\frac{d^n}{dx^n}G(x)-\frac{d^n}{dx^n}G(-x)\right)=0
\end{equation*}
for $n\geq 0$.

Let $W(x):=(\cdot \text{ mod }1)_{\ast}(F)(x).$ Then for $x\in \left[0,\frac{1}{2}\right)$ we have
$$
W(x)=F(x)+\sum_{k=1}^{\infty}F(x+k)=F(x)+G(x)
$$
and
\begin{equation*}
W(-x)=\sum_{k\in\Z}F(-x+k)=\sum_{k=1}^{\infty}F(-x+k)=G(-x).
\end{equation*}
Hence we obtain
\begin{multline*}
\lim_{x\to 0^+}x^{-a}(W(x)-W(-x))\\
=\lim_{x\to 0^+}x^{-a}F(x)+\lim_{x\to 0^+}x^{-a}(G(x)+G(-x))=\lim_{x\to 0^+}x^{-a}F(x)=A_0.
\end{multline*}
In the remaining part of the proof we will argue by induction. Suppose that for some $k\geq 1$ we already know $A_0,A_1,\dots,A_{k-1}$. We will find now $A_k$. Let
\begin{multline*}
\overline{A}_k=\lim_{x\to 0^+}\left(x^{-a}\left(\frac{d^{k}}{dx^{k}}W(x)-\frac{d^{k}}{dx^{k}}W(-x)\right)-x^{-a}\frac{d^{k}}{dx^{k}}\left(x^{a}\cdot\sum_{n=0}^{k-1}A_nx^n \right) \right).
\end{multline*}
By~\eqref{eq:conti} and by analicity of the function $\widetilde{F}$ at zero we obtain
\begin{multline*}
\overline{A}_k=\lim_{x\to 0^+}x^{-a}\left(\frac{d^{k}}{dx^{k}} F(x)-\frac{d^{k}}{dx^{k}}\left(x^{a}\cdot\sum_{n=0}^{k-1}A_nx^n\right)\right.\\
\left.+\frac{d^{k}}{dx^{k}}G(x)-\frac{d^{k}}{dx^{k}}G(-x)\right)\\
=\lim_{x\to 0^+}x^{-a}\left(\frac{d^{k}}{dx^{k}} \left( x^{a}\cdot\sum_{n=k}^{\infty}A_nx^n\right) \right)\\
=\lim_{x\to 0^+}\sum_{n=k}^{\infty}\left(n+a\right)\cdot \left(n+a-1\right)\cdot \ldots \cdot \left(n+a-k+1\right)A_nx^{n-k}\\
=(k+a)\cdot(k-1+a)\cdot\ldots\cdot (1+a) \cdot A_{k}=\frac{\Gamma(1+a)}{\Gamma(k+1+a)}{A}_k.
\end{multline*}
Hence $A_k=\frac{\Gamma(1+a)}{\Gamma(k+1+a)}\overline{A}_k$, which ends the proof.
\end{proof}

\begin{uw}
In the proof of the above proposition we obtain the following equality:
\begin{equation*}
x^{-a}\frac{d^{k}}{dx^{k}}\left(x^{a}\cdot\sum_{n=0}^{k-1}A_nx^n \right)=\sum_{r=1}^{k}A_{k-r}\frac{\Gamma(k+a-r+1)}{\Gamma(a-r+1)}x^{-r}.
\end{equation*}
\end{uw}

\begin{pr}\label{pr:nawijanie2}
Assume that the functions $F_1\in L^1(0,\infty)$ and $F_2\in L^1(-\infty,0)$ are such that:
\begin{itemize}
\item
in the right-hand side neighbourhood of zero the function $F_1$ is of the form $F_1(x)=x^{a}\cdot \widetilde{F}_1(x)$, where $a>-1$ and $\widetilde{F}_1$ is analytic at zero with the series expansion in the right-hand side neighbourhood of zero of the form $\widetilde{F}_1(x)=\sum_{n=0}^{\infty}B_nx^n$,
\item
$F_2$ is analytic at zero with the series expansion in the left-hand side neighbourhood of zero of the form $F_2(x)=\sum_{n=0}^{\infty}C_nx^n$,
\item
the functions $F_1$ and $F_2$ are smooth and their derivatives decay ``sufficiently fast'',
\item
all the derivatives of the function $F_1\ast F_2$ decay ``sufficiently fast''.
\end{itemize}
Then the function $(\cdot \text{ mod }1)_{\ast}(F_1\ast F_2)$ determines the value of the expressions
\begin{equation}\label{eq:wsprozne}
\widetilde{A}_n:=-\sum_{k=0}^{n}B_{n-k}C_{k}\frac{\Gamma(k+1)\Gamma(a+n-k+1)}{\Gamma(n+a+2)}
\end{equation}
for $n\geq 1$.
\end{pr}

\begin{proof}
The proof will be similar to the proof of Proposition~\ref{pr:nawijanie1}. As before, without loss of generality we may assume that $a\leq 0$.

Let $G\colon \left(-\frac{1}{2},\frac{1}{2}\right)\to\R$ be given by the formula
\begin{equation*}
G(x)=\sum_{k\neq 0,k\in\Z}(F_1\ast F_2)(x+k).
\end{equation*}

By Lemma~\ref{lm:pomo}
\begin{equation}\label{eq:conti}
\lim_{x\to 0}\left(\frac{d^n}{dx^n}G(x)-\frac{d^n}{dx^n}G(-x)\right)=0
\end{equation}
for $n\geq 0$.

Let 
\begin{equation*}
W(x):=(\cdot \text{ mod }1)_{\ast}(F_1\ast F_2)(x)=(F_1\ast F_2)(x)+G(x).
\end{equation*}
Then for $k\geq 0$ we have
\begin{multline}\label{eq:raz}
\lim_{x\to 0^+}\left(\left(\frac{d^k}{dx^k}W(x)-\frac{d^k}{dx^k}W(-x)\right)\right.\\
\left.-\left(\frac{d^k}{dx^k}(F_1\ast F_2)(x)+\frac{d^k}{dx^k}(F_1\ast F_2)(-x)\right)\right)=0.
\end{multline}
For $x>0$, by Proposition~\ref{stw:deriv} we have
\begin{multline}\label{eq:dwa}
\frac{d^k}{dx^k}(F_1\ast F_2)(x)-\frac{d^k}{dx^k}(F_1\ast F_2)(-x)\\
=\int_{x}^{\infty}F_1(y)\left(\frac{d^k}{dx^k}F_2(x-y)-\frac{d^k}{dx^k}F_2(-x-y) \right)\ dy\\
-\int_{0}^{x}F_1(y)\frac{d^k}{dx^k}F_2(-x-y)\ dy- \sum_{l=0}^{k-1}\frac{d^{k-l-1}}{dx^{k-l-1}}F_1(x) \frac{d^l}{dx^l}F_2(0).
\end{multline}
By uniform continuity of the function $\frac{d^k}{dx^k}F_2$ and by the integrability of the function $F_1$ we obtain
\begin{equation}\label{eq:trzy}
\lim_{x\to 0^+}\left(\int_{x}^{\infty}F_1(y)\left(\frac{d^k}{dx^k}F_2(x-y)-\frac{d^k}{dx^k}F_2(-x-y) \right)\ dy\right)=0. 
\end{equation}
Analogously, by analicity of the function $F_2$ at zero and by the integrability of $F_1$ we have
\begin{equation}\label{eq:cztery}
\lim_{x\to 0^+}\left(\int_{0}^{x}F_1(y)\frac{d^k}{dx^k}F_2(-x-y)\ dy\right)=0.
\end{equation}

Let $\widetilde{A}_0=0$ and suppose that for some $k\geq 1$ we already know $\widetilde{A}_0,\dots,\widetilde{A}_{k-1}$. We will show how to find $\widetilde{A}_k$. Consider first the case $a=0$. Let
\begin{equation*}
\overline{A}_k=\lim_{x\to 0^+}\left(\frac{d^k}{dx^k}W(x)-\frac{d^k}{dx^k}W(-x) \right)
\end{equation*}
Using~\eqref{eq:raz},~\eqref{eq:dwa},~\eqref{eq:trzy} and~\eqref{eq:cztery} we obtain
\begin{multline*}
\overline{A}_k=\lim_{x\to 0^+}\left(- \sum_{l=0}^{k-1}\frac{d^{k-l-1}}{dx^{k-l-1}}F_1(x) \frac{d^l}{dx^l}F_2(0)\right)\\
=-\sum_{l=0}^{k-1}B_{k-l-1}C_{l}\Gamma(k-l)\Gamma(l+1)\\
=-\sum_{l=0}^{k-1}B_{k-l-1}C_l\frac{\Gamma(l+1)\Gamma(k-l)}{\Gamma(k+1)}\cdot\Gamma(k+1)=\widetilde{A}_k\cdot\Gamma(k+1).
\end{multline*}
Hence $\widetilde{A}_k=\frac{1}{\Gamma(k+1)}\cdot \overline{A}_k$.

Let now $a\neq 0$ and let
\begin{multline*}
\overline{A}_{k}=\lim_{x\to 0^+}\left(x^{-a}\cdot\left(\frac{d^{k}}{dx^{k}}W(x)-\frac{d^{k}}{dx^{k}}W(-x)\right)\right.\\
\left.-\sum_{r=1}^{k}\widetilde{A}_{k-r}\frac{\Gamma(k+a-r+1)}{\Gamma(a-r+1)} x^{-r} \right).
\end{multline*}
Using~\eqref{eq:raz},~\eqref{eq:dwa},~\eqref{eq:trzy} and~\eqref{eq:cztery} we obtain
\begin{multline*}
\overline{A}_k=\lim_{x\to 0^+}\left(x^{-a}\cdot\left(- \sum_{l=0}^{k-1}\frac{d^{k-l-1}}{dx^{k-l-1}}F_1(x) \frac{d^l}{dx^l}F_2(0)\right)\right.\\
\left.-\sum_{r=1}^{k}\widetilde{A}_{k-r}\frac{\Gamma(k+1-r+1)}{\Gamma(a-r+1)} x^{-r} \right)\\
=\lim_{x\to 0^+}\left(x^{-a}\cdot\left(- \sum_{l=0}^{k-1}\frac{d^{l}}{dx^{l}}F_1(x) \frac{d^{k-l-1}}{dx^{k-l-1}}F_2(0)\right)\right.\\
\left.-\sum_{r=1}^{k}\widetilde{A}_{k-r}\frac{\Gamma(k+a-r+1)}{\Gamma(a-r+1)} x^{-r} \right).
\end{multline*}
For $x>0$ sufficiently small, $s$-th derivative of the function $F_1$ is given by the formula
$$
\frac{d^s}{dx^s}F_1(x)=\sum_{n=0}^{\infty}B_n\frac{\Gamma(n+a+1)}{\Gamma(n+a-(s-1))}x^{n+a-s}.
$$
Therefore
\begin{multline}\label{eq:starr}
\overline{A}_k=\lim_{x\to 0^+}\left(-\sum_{l=0}^{k-1} \left(\sum_{n=0}^{\infty}B_n\frac{\Gamma(n+a+1)}{\Gamma(n+a-(l-1))}x^{n-l} \right)\Gamma(k-l)C_{k-l-1}\right.\\
\left.- \sum_{r=1}^{k}\widetilde{A}_{k-r}\frac{\Gamma(a+1)}{\Gamma(a-r+1)}x^{-r}\right).
\end{multline}
The coefficient in front of $x^{-r}$ in the expression
$$
-\sum_{l=0}^{k-1} \left(\sum_{n=0}^{\infty}B_n\frac{\Gamma(n+a+1)}{\Gamma(n+a-(l-1))}x^{n-l} \right)\Gamma(k-l)C_{k-l-1}
$$
is equal to zero for $r\geq k$ and for $1\leq r\leq k-1$ it is equal to
\begin{multline*}
-\sum_{l=r}^{k-1}B_{l-r}C_{k-l-1}\frac{\Gamma(l-r+a+1)\Gamma(k-l)}{\Gamma(-r+a+1)}\\
=-\sum_{l=0}^{k-r-1}B_{k-r-l-1}C_l\frac{\Gamma(k-r-l+a)\Gamma(l+1)}{\Gamma(k-r+a+1)}\cdot\frac{\Gamma(k-r+a+1)}{-r+a+1}\\
=\widetilde{A}_{k-r}\cdot\frac{\Gamma(k-r+a+1)}{\Gamma(-r+a+1)}
\end{multline*}
Since for $k=r$ we have $\widetilde{A}_{k-r}\cdot\frac{\Gamma(k+a-r+1)}{\Gamma(a-r+1)}x^{-r}=0=\widetilde{A}_0$, by~\eqref{eq:starr} we obtain further
$$
\overline{A}_k=\widetilde{A}_k\cdot\frac{\Gamma(k+a+1)}{\Gamma(a+1)},
$$
whence $\widetilde{A}_k=\frac{\Gamma(a+1)}{\Gamma(k+a+1)}\cdot \overline{A}_k$, which ends the proof.
\end{proof}

\begin{uw}\label{uw:na2}
Let the functions $F_1, F_2\in L^1(0,\infty)$ be such that:
\begin{itemize}
\item
in the right-hand side neighbourhood of zero $F_1(x)=x^a\widetilde{F}_1(x)$, where $\widetilde{F}_1$ is analytic at zero and $\widetilde{F}_1(x)=\sum_{n=0}^{\infty}B_nx^n$ in the right-hand side neighbourhood of zero,
\item
$F_2$ is analytic at zero and $F_2(x)=\sum_{n=0}^{\infty}C_nx^n$ in the right-hand side neighbourhood of zero,
\item
$F_1$ and $F_2$ are smooth,
\item
all the derivatives of the functions $F_1,F_2$ and $F_1\ast F_2$ decay ``sufficiently fast''.
\end{itemize}
Then the numbers
\begin{equation}\label{eq:aeny}
A_n=\sum_{k=0}^{n}C_kB_{n-k}\frac{\Gamma(k+1)\Gamma(a+n-k+1)}{\Gamma(n+a+2)}
\end{equation}
are the coefficients in the series expansion of $F_1\ast F_2$ in the right-hand side neighbourhood of zero:
\begin{equation*}
F_1\ast F_2(x)=\sum_{n=0}^{\infty}A_nx^{n+a+1}.
\end{equation*}
\end{uw}

\begin{wn}\label{wn:222}
Let the functions $F_1$ and $F_2$ be such that:
\begin{itemize}
\item
$F_1,F_2 \in L^1(0,\infty)$ or $F_1\in L^1(-\infty,0)$, $F_2\in L^1(0,\infty)$,
\item
in the neighbourhood of zero (right-hand side or left-hand side - depending on the domain) $F_1(x)=x^a\widetilde{F}_1(x)$, where $a>-1$, the function $\widetilde{F}_1$ is analytic at zero with the series expansion around zero given by $\widetilde{F}_1(x)=\sum_{n=0}^{\infty}B_nx^n$,
\item
$F_2$ is analytic at zero with the series expansion around zero given by $F_2(x)=\sum_{n=0}^{\infty}C_nx^n$,
\item
$F_1$ and $F_2$ are smooth,
\item
all the derivatives of the functions $F_1,F_2$ and $F_1\ast F_2$ decay ``sufficiently fast''.
\end{itemize}
Then $(\cdot \text{ mod }1)_{\ast}(F_1\ast F_2)$ determines a pair of sequences $\left\{(A_n)_{n\in\N},(-A_n)_{n\in\N} \right\}$,
where $A_n$ for $n\geq 0$ is given by formula~\eqref{eq:aeny}.
\end{wn}
\begin{proof}
Notice that the proofs of Propositions~\ref{pr:nawijanie1} and~\ref{pr:nawijanie2} go along the same lines. We argue by induction - in the consecutive steps for $k\in\N$ we calculate the right-hand side limit as zero of the expressions of the following form:
\begin{equation}\label{eq:je}
x^{-a}\left(\frac{d^k}{dx^k}W(x)-\frac{d^k}{dx^k}W(-x)\right)-\sum_{r=1}^{k}{A}_{k-r}\frac{\Gamma(k+a-r+1)}{\Gamma(a-r+1)} x^{-r} ,
\end{equation} 
where $W(x)=\sum_{k\in \Z}F(x+k)$ (in Proposition~\ref{pr:nawijanie1}) or
\begin{equation}\label{eq:jeje}
x^{-a}\left(\frac{d^k}{dx^k}W(x)-\frac{d^k}{dx^k}W(-x)\right)-\sum_{r=1}^{k}\widetilde{A}_{k-r}\frac{\Gamma(k+a-r+1)}{\Gamma(a-r+1)} x^{-r} ,
\end{equation}
where $W(x)=\sum_{k\in \Z}F_1\ast F_2(x+k)$ (in Proposition~\ref{pr:nawijanie2}).

Using Proposition~\ref{pr:nawijanie1} for the functions $F=F_1\ast F_2$, by Remark~\ref{uw:na2} it follows that as the limit of the expression~\eqref{eq:je} one obtains the consecutive numbers $A_n$, whereas as the limit of the expression~\eqref{eq:jeje}, the consecutive numbers $-A_n$. Therefore, independent of whether $F_1\in L^1(-\infty,0)$ or $F_1\in L^1(0,\infty)$, the function $(\cdot \text{ mod }1)_{\ast}(F_1\ast F_2)$ determines the pair of sequences $\left\{(A_n)_{n\in\N},(-A_n)_{n\in\N} \right\}$.
\end{proof}

\begin{uw}\label{uw:222}
The assertion of Corollary~\ref{wn:222} remains true also in the following situations (the rest of the assumptions and the notation remain unchanged):
\begin{itemize}
\item
$F_1\in L^1(-\infty,0)$, $F_2\in L^1(0,\infty)$,
\item
$F_1,F_2 \in L^1(-\infty,0)$.
\end{itemize}
\end{uw}

\paragraph{Application for the rescaled functions}
\begin{wn}\label{wn:wnnnn}
Let $d\geq 1$ be even, $t_1,\dots,t_d\in\R\setminus\{0\}$ and let the function $F\in L^1(0,\infty)$ be such that:
\begin{itemize}
\item
in the right-hand side neighbourhood of zero $F(x)=x^{-1/2}\widetilde{F}(x)$, where the function $\widetilde{F}$ is analytic at zero with the series expansion around zero given by $\widetilde{F}(x)=\sum_{n=0}^{\infty}B_nx^n$,
\item
the function $F$ is smooth,
\item
for any $1\leq k\leq d$ and for any subset $\{s_1,\dots,s_k\}\subset \{t_1,\dots,t_d\}$ such that $\{s_1,\dots,s_k\}\subset (0,\infty)$ or $\{s_1,\dots,s_k\}\subset (-\infty,0)$
all the derivatives of the function $F_{s_1}\ast \dots \ast F_{s_k}$ decay ``sufficiently fast'',
\item
all the derivatives of the function $F_{t_1}\ast \dots \ast F_{t_d}$ decay ``sufficiently fast''.
\end{itemize}
Then the function $(\cdot \text{ mod }1)_{\ast}(F_{t_1}\ast\dots\ast F_{t_d})$ determines a pair of sequences $\left\{(A_n)_{n\in\N},(-A_n)_{n\in\N} \right\}$,
where the numbers $A_n$, $n\geq 0$, are given by
$$
A_n:=\frac{1}{\sqrt{|t_1\cdot\ldots\cdot t_d|}}\cdot\sum_{k_1+\dots+k_d=n}\prod_{i=1}^{d}\frac{a_{k_i}\Gamma(k_i+1/2)}{t_i^{k_i}}\frac{1}{\Gamma(n+d/2)}.
$$
\end{wn}
\begin{proof}
Without loss of generality we may assume that $t_1\leq t_2\leq \dots \leq t_d$. Let $0\leq i_0\leq d$ be such that $t_1\leq \dots \leq t_{i_0}<0<t_{i_0+1}\leq \dots \leq t_d$. To fix attention, assume that $i_0$ is even (in the case where $i_0$ is odd the proof is similar). Using the formulas~\eqref{eq:wzorparz} and~\eqref{eq:wzorparz2} in Corollary~\ref{wn:9} we find the series expansions of the function $h_{t_1}\ast \dots \ast h_{t_{i_0}}$ and $h_{t_{i_0+1}}\ast \dots \ast h_{t_d}$. Denote the coefficients appearing in these equations by $(B_n)_{n\in \N}$ and $(C_n)_{n\in \N}$ respectively. By Corollary~\ref{wn:222} (and by Remark~\ref{uw:222}) it follows that we can determine the consecutive elements of the sequence $(A_n)_{n\in\N}$ or $(-A_n)_{n\in\N}$, where
\begin{equation*}
A_n=\sum_{k=0}^{n}C_kB_{n-k}\frac{\Gamma(k+1)\Gamma(a+n-k+1)}{\Gamma(n+a+2)}.
\end{equation*}
The claim follows by Remark~\ref{uw:na2} and formula~\eqref{eq:rdrd} in Corollary~\ref{wn:9}.\footnote{Remark~\ref{uw:na2} is not necessary for completing this proof. Once we have determined the $(B_n)_{n\in\N}$ and $(C_n)_{n\in\N}$ it can be checked directly that the numbers $(A_n)_{n\in\N}$ in this case are (up to the sign) the same as the coefficients in the expansion~\eqref{eq:rdrd}.}
\end{proof}

\begin{wn}\label{wn:wielomianytu}
Under the assumptions of the above corollary, assuming additionally that $a_0\neq 0$,  the function $(\cdot \text{ mod }1)_{\ast}(F_{t_1}\ast\dots\ast F_{t_d})$ determines the sequence $(B_n)_{n\in\N}$ given by
\begin{equation}\label{eq:liczbybn}
B_n=\sum_{k_1+\dots+k_d=n}\prod_{i=1}^{d}\frac{a_{k_i}\Gamma(k_i+1/2)}{t_i^{k_i}}.
\end{equation}
\end{wn}
\begin{proof}
Notice that 
$$
A_0=\frac{1}{\sqrt{|t_1\cdot\ldots\cdot t_d|}}\cdot a_0^d\cdot (\Gamma(1/2))^d\cdot \frac{1}{\Gamma(d/2)}.
$$
Knowing the function $F$, we also know the number $a_0$, which gives us the sign of $A_0$. By Corollary~\ref{wn:wnnnn} the sequence $(A_n)_{n\in\N}$ and thus also $(B_n)_{n\in\N}$ is determined.
\end{proof}
Notice that the assumptions that $a_0\neq 0$ in the above corollary can be weakened. It suffices to find the largest  $i_0\in\N$ such that $a_0=a_1=\dots=a_{i_0}$ i check the sign of $a_{i_0+1}$.

\subsubsection{Fast decay of all of the derivatives}\label{se:5.2.4}

In the previous section an important role was played by functions whose derivatives decay ``sufficiently fast''. We will now introduce a property which will be easier to check and which will be preserved under taking convolutions. It will allow us to formulate a criterion ensuring the needed ``sufficienlty fast'' decay.

Recall that the conditions $\mathcal{W},\mathcal{W}_1,\mathcal{W}_2$ describe the ``sufficiently fast'' decay (see Definition~\ref{df:szybkispadek} on page~\pageref{df:szybkispadek}).

We will write  $\mathcal{W}(F)$ for a function $F$ enjoying property $\mathcal{W}$. To underline which constants $A>1$, $t>0$ and $r\in(-1,0)$ we mean we will sometimes write $\mathcal{W}(F,A)$, $\mathcal{W}(F,t)$, $\mathcal{W}(F,A,r)$, $\mathcal{W}(F,A,t)$ or $\mathcal{W}(F,A,r,t)$. Sometimes we will also write $\mathcal{W}_1(F)$ or $\mathcal{W}_2(F)$ to express the fact that one of the conditions from the definition of the property $\mathcal{W}$ holds.

\begin{lm}\label{stw:jjj}
Let the function $F\colon (-\infty,0)\cup (0,\infty)\to \R$ be such that for some $n\in\N$ the function $x\mapsto x^n \cdot F(x)$ has property $\mathcal{W}_1$. Then for some $x_0>0$ there exists a non-increasing function $H\colon (x_0,\infty)\to\infty$  such that 
\begin{equation*}
H\in L^1(x_0,\infty)\text{ and }|F(x)|<H(|x|)\text{ for }|x|>x_0. 
\end{equation*}
In particular, if the function $F$ is smooth and all its derivatives, i.e. $\frac{d^k}{dx^k}F$, are such that the functions $x\mapsto x^{n_k}\cdot \frac{d^k}{dx^k}F(x)$ have property $\mathcal{W}_1$ for some $n_k\in\N$, then all the derivatives of the function $F$ decay ``sufficiently fast''.\footnote{The definition of the function whose all the derivatives decay ``sufficiently fast'' can be found in Section~\ref{p343}, on page~\pageref{def:dostszyb}.}
\end{lm}
\begin{proof}
For $H$ it suffices to take $x\mapsto Ae^{-\frac{|x|}{A}}$ (the number $A$ is the same as in the definition of property $\mathcal{W}$, for $x_0$ we put $\max\{t,1\}$, where $t$ is as in the definition of property $\mathcal{W}$). 
\end{proof}

\begin{uw}
For $1<A<B$, $-1<r<s<0$ and $x\neq 0$ the following inequalities hold:
\begin{itemize}
\item
$e^{-\frac{|x|}{A}}<e^{-\frac{|x|}{B}}$,
\item
$|x|^s=|x|^r\cdot |x|^{s-r}<|x|^r\cdot t^{s-r}$ for $0<|x|<t$.
\end{itemize}
Some of the properties following from the above inequalities will be important for us.
\begin{itemize}
\item
In order to show that the function $F$ has property $\mathcal{W}$, it suffices to find constants $t>0$, $r\in (-1,0)$ and $A_1,A_2,A_3>1$ such that $|g(x)|<A_1 e^{-\frac{x}{A_2}}$ for $x>t$ and $|g(x)|<A_3x^r$ for $0<x<t$. 
\item
Moreover, given a finite family of functions $F_1,\dots,F_m$ ($m\geq 1$) enjoying property $\mathcal{W}$, there exist numbers $A>1$, $r\in(-1,0)$ and $t>0$ such $\mathcal{W}(F_i,A,r,t)$ holds for $1\leq i\leq m$ (a constant $t>0$ common for all the functions exists by their boundedness on compact subsets of $(-\infty,0)\cup(0,\infty)$).
\end{itemize}
\end{uw}

\begin{uw}
If there exists $A>1$ such that for $0<x<t$ we have $|F(x)|<A$, then for $r\in(-1,0)$ there exists $B>1$ such that for $0<x<t$ we have $|F(x)|<Bx^r$.
\end{uw}

\begin{uw}\label{uw:nowauwaga}
If $\mathcal{W}(F)$ holds, then for any $c>0$ the function $F$ is bounded in the set $(-\infty,-c]\cup[c,\infty)$.
\end{uw}

\begin{uw}\label{uw:przemnozyc}
By the convergence $\lim_{x\to\infty} x^n \cdot e^{-\frac{x}{2M}}=0$ for $n,M\in\N$ it follows that whenever $\mathcal{W}_1(F,t)$ holds for some function $F$ and some $t>0$, then also $\mathcal{W}_1(x\mapsto x^n \cdot F(x),t)$ holds for any $n\in\N$.
\end{uw}

\begin{uw}\label{uw:3.21}
Let $F\in L^1(0,\infty)$ be such that for some $n\in\N$ we have $\mathcal{W}(x\mapsto x^n\cdot F(x))$. Then $\lim_{x\to\infty}F(x)=0$ and for any $c>0$ the function $F$ is uniformly continuous on the interval $[c,\infty)$.
\end{uw}

\begin{lm}\label{lm:lemat5}
Let $m\geq 1$ and let $F_1,\dots, F_m \in L^1(0,\infty)$. Assume that $\mathcal{W}(F_i)$ holds for $1\leq i \leq m$. Then also $\mathcal{W}(F_1\ast \dots \ast F_m)$ holds.
\end{lm}

\begin{proof}
Let $F_1,\dots,F_m\in L^1(0,\infty)$ satisfy the assumptions. We will show first that $\mathcal{W}(F_1\ast F_2)$ holds. Let $A>1$, $r\in(-1,0)$, $t_1,t_2>0$ be such that
\begin{equation}\label{eq:as1}
\mathcal{W}(F_1,A,r,t_1) \text{ and }\mathcal{W}(F_2,A,r,t_2)\text{ hold.}
\end{equation}
We claim that $\mathcal{W}(F_1\ast F_2,B,r,t_1+t_2)$ holds for some $B>1$. We will show first that $\mathcal{W}_1(F_1\ast F_2,B,r,t_1+t_2)$ holds for some $B>1$. Take $x> t_1+t_2$. We have

\begin{align}
|(F_1\ast& F_2)(x)|=\left|\int_{0}^{t_1}F_1(y)F_2(x-y)\ dy+\int_{x-t_2}^{x}F_1(y)F_2(x-y)\ dy \right. \nonumber\\
&\left.+\int_{t_1}^{x-t_2}F_1(y)F_2(x-y)\ dy \right| \nonumber\\
\leq&\int_0^{t_1}\left| F_1(y)F_2(x-y) \right|\ dy+ \int_{0}^{t_2}\left| F_1(x-y)F_2(y)\right|\ dy \nonumber\\
&+\int_{t_1}^{x-t_2}\left| F_1(y)F_2(x-y)\right|\ dy.\label{eq:rozb}
\end{align}

For $0<y<t_1$ we have $x-y>x-t_1> t_2$, therefore by~\eqref{eq:as1} we obtain
\begin{equation*}
|F_1(y)|<Ay^r\text{ and }|F_2(x-y)|<Ae^{-\frac{x-y}{A}},
\end{equation*}
whence
\begin{multline}\label{eq:eins}
\int_0^{t_1}|F_1(y)F_2(x-y)|\ dy\\
<A^2e^{-\frac{x}{A}}\int_{0}^{t_1}e^{\frac{y}{A}}y^r\ dy
\leq A^2e^{\frac{t_1}{A}}\frac{t_1^{r+1}}{r+1}e^{-\frac{x}{A}}=Me^{-\frac{x}{A}}
\end{multline}
for some $M>0$. The middle term in~\eqref{eq:rozb} can be bounded in a similar way by $Me^{-\frac{x}{A}}$ for some $M>0$. We will find now an upper bound for the last summand. By~\eqref{eq:as1} and since $x/2\cdot e^{-x}<e^{-x/2}$ for $x>0$, we have
\begin{multline}\label{eq:zwei}
\int_{t_1}^{x-t_2}|F_1(y)F_2(x-y)|\ dy<A^2\int_{t_1}^{x-t_2}e^{-\frac{y}{A}}e^{-\frac{x}{A}+\frac{y}{A}}\ dy\\
=A^2 e^{-\frac{x}{A}}(x-t_1-t_2)<A^2e^{-\frac{x}{A}}x=2A^3e^{-\frac{x}{A}}\frac{x}{2A}<2A^3e^{-\frac{x}{2A}}.
\end{multline}
By~\eqref{eq:eins} and~\eqref{eq:zwei} we obtain
\begin{equation*}
|(F_1\ast F_2)(x)| <Be^{-\frac{x}{B}}\text{ for }x>t_1+t_2
\end{equation*}
for some $B>1$. Therefore $\mathcal{W}_1(F_1\ast F_2,B,r,t_1+t_2)$ indeed holds.

We will show now that $\mathcal{W}_2(F_1\ast F_2,t_1+t_2)$ also holds. Let now $0<x<t_1+t_2$. 
By the symmetry of the investigated expressions, without loss of generality we may assume that $t_1\leq t_2$. If $x> t_2$ we have
\begin{align}
|(F_1\ast& F_2) (x)| \leq\int_{0}^{x-t_2}|F_1(y)F_2(x-y)|\ dy + \int_{x-t_2}^{t_1}|F_1(y)F_2(x-y)|\ dy\nonumber \\
&+ \int_{t_1}^{x}|F_1(y)F_2(x-y)|\ dy=\int_{t_2}^{x}|F_1(x-y)F_2(y)|\ dy \nonumber \\
&+ \int_{x-t_2}^{t_1}|F_1(y)F_2(x-y)|\ dy+ \int_{t_1}^{x}|F_1(y)F_2(x-y)|\ dy.\label{strr}
\end{align}
The summands $\int_{t_2}^{x}|F_1(x-y)F_2(y)|\ dy$ and $\int_{t_1}^{x}|F_1(y)F_2(x-y)|\ dy$ are of the same form and they can be bounded from above in the same way. We will give now an upper bound for the first of them. Since we have $\mathcal{W}_1(F_2,A,t_2)$, hence for $y\in (t_2,x)$ we have $|F_2(y)|<Ae^{-\frac{y}{A}}$. Since the function $y\mapsto Ae^{-\frac{y}{A}}$ is bounded on the set $(0,t_1+t_2)$, and the function $F_1$ is integrable, therefore for some $M\in\R$
\begin{equation}\label{strrr}
\int_{t_2}^{x}|F_1(x-y)F_2(y)|\ dy<M.
\end{equation}
In a similar way, for some $M\in\R$
\begin{equation}\label{strrr2}
\int_{t_1}^{x}|F_1(y)F_2(x-y)|\ dy<M.
\end{equation}
Moreover, since we have $\mathcal{W}_2(F_1,A,t_1)$ and $\mathcal{W}_2(F_2,A,t_2)$, we obtain
\begin{multline}\label{strrr3}
 \int_{x-t_2}^{t_1}|F_1(y)F_2(x-y)|\ dy\\
 \leq A^2\int_{x-t_2}^{t_1}y^r (x-y)^r \ dy  < A^2\int_{0}^{x} y^r (x-y)^r\ dy\\
 =A^2 \frac{\Gamma(r+1)\Gamma(r+1)}{\Gamma(2r+2)}x^{2r+1}<A^2 \frac{\Gamma(r+1)\Gamma(r+1)}{\Gamma(2r+2)}(t_1+t_2)^{r+1}x^r
\end{multline}
(the equality in the second line of the above calculations holds by Remark~\ref{lm:calki}).
By the estimates~\eqref{strr},~\eqref{strrr},~\eqref{strrr2} and~\eqref{strrr3} it follows that there exists $C>0$ such that 
$$
|(F_1\ast F_2)(x)|<C x^{r}\text{ for }t_2<x<t_1+t_2.
$$

If $t_1\leq x<t_2$ we have
\begin{multline*}
|F_1\ast F_2(x)|\leq \int_{0}^{x}|F_1(y)F_2(x-y)|\ dy\\
= \int_0^{t_1} |F_1(y)F_2(x-y)|\ dy + \int_{t_1}^{x}|F_1(y)F_2(x-y)|\ dy.
\end{multline*}
By treating both summands in the above inequality in a similar way as in the case $x> t_2$ one can deduce that there exists $C>0$ such that 
$$
|(F_1\ast F_2)(x)|<C x^{r}\text{ for }t_1<x<t_2.
$$
For $0<x<t_1$ the situation is even simpler and we have
\begin{equation*}
|F_1\ast F_2(x)|\leq \int_{0}^{x} |F_1(y)F_2(x-y)|\ dy.
\end{equation*}
Finding again similar bounds as in the case $x> t_2$, we obtain
\begin{equation*}
|(F_1\ast F_2)(x)|<C x^{r}\text{ for }x<t_1
\end{equation*}
for some $C>0$. This means that $\mathcal{W}_2(F_1\ast F_2,t_1+t_2)$ holds. Therefore also $\mathcal{W}(F_1\ast F_2)$ holds. By induction on $m$ we obtain that $\mathcal{W}(F_1\ast \dots \ast F_m)$ also holds.
\end{proof}

\begin{pr}\label{lm:wlasnoscPdlasplotu}
Let $d\geq 1$ and $F_1,\dots,F_d\in L^1(0,\infty)$. Let $k\geq 0$. Assume that for any $1\leq i,j \leq d$ the function $F_i\ast F_j$ is analytic at zero and we have $\mathcal{W}\left(x \mapsto x^l\cdot\frac{d^l}{dx^l}F_i(x)\right)$ for $0\leq l \leq k$ and $1\leq i\leq d$. Then we also have $\mathcal{W}\left(x \mapsto x^k\cdot \frac{d^k}{dx^k}(F_1\ast \dots \ast F_d)(x)\right)$.
\end{pr}

\begin{proof}
For $k=0$ the claim follows directly by Lemma~\ref{lm:lemat5}. The proof for $k\geq 1$ will be devided into two steps. We will show first that $\mathcal{W}\left(\frac{d^l}{dx^l}F_{i_1}\ast F_{i_2}\right)$ holds for $1\leq i_1<i_2\leq d$ and $1\leq l \leq k$. The second part will consist of an inductive argument using the first part of the proof.

\vspace*{10pt}
\noindent
\emph{Step 1.} We will show first that $\mathcal{W}\left(\frac{d^l}{dx^l}(F_{i_1}\ast F_{i_2}) \right)$ holds for any $1\leq i_1<i_2\leq m$ and $1\leq l\leq k$. Without loss of generality we may take $i_1=1$ and $i_2=2$. Fix $1\leq l\leq k$. By Proposition~\ref{uw:uw2} we have
\begin{multline}\label{eq:toge}
\left|\frac{d^l}{dx^l}(F_1\ast F_2)(x)\right|\leq \int_0^{\frac{x}{2}} |F_1(y)|\left| \frac{d^l}{dx^l} F_2(x-y)\right|\ dy\\
+\int_0^{\frac{x}{2}} |F_2(y)|\left| \frac{d^l}{dx^l} F_1(x-y)\right|\ dy\\
+\sum_{j=0}^{l-1} |w_j| \left| \frac{d^j}{dx^j}F_1\left(\frac{x}{2} \right)\right| \left|\frac{d^{l-1-j}}{dx^{l-1-j}}F_2\left(\frac{x}{2} \right)\right|
\end{multline}

for some $w_j\in\R$, $0\leq j\leq l-1$.

Let $t>0$ and $A>1$ be such that for $0\leq j\leq l-1$ we have $\mathcal{W}\left(x\mapsto x^j \cdot \frac{d^j}{dx^j}F_1(x),t/2,A\right)$ and $\mathcal{W}\left(x\mapsto x^j \cdot \frac{d^j}{dx^j}F_2(x),t/2,A\right)$. 

Using similar arguments as in the proof of Lemma~\ref{lm:lemat5} one can show that $\mathcal{W}_1\left(\frac{d^l}{dx^l}(F_1\ast F_2) \right)$ holds.

Since the function $F_1\ast F_2$ is analytic at zero, the derivative $\frac{d^l}{dx^l}(F_1\ast F_2)$ is also is analytic at zero and has a finite limit $\lim_{x\to 0^+}\frac{d^l}{dx^l}(F_1\ast F_2)(x)$. Therefore, as a continuous function, it is uniformly continuous on the interval $(0,t]$. Therefore there exists a constant $A>1$ such that $\left|\frac{d^l}{dx^l}(F_1\ast F_2)\right|<Ax^{r}$. Hence  $\mathcal{W}_2\left(\frac{d^l}{dx^l}(F_1\ast F_2) \right)$ holds. Thus we have shown that $\mathcal{W}(\frac{d^l}{dx^l}\left(F_1\ast F_2)\right)$ holds for any $0\leq l\leq k$.

\vspace*{10pt}
\noindent
\emph{Step 2.}
Now we will use an inductive argument. Assume that for some $1\leq n_0< d$ for any $1\leq n\leq n_0$ and any $1\leq i_1<\dots<i_n\leq d$ and $1\leq l\leq k$ we have 
$$
\mathcal{W}\left(x\mapsto x^l \cdot\frac{d^l}{dx^l}\left(F_{i_1}\ast \dots \ast F_{i_{n}}\right)(x)\right).
$$
We will show that for any $1\leq i_1<\dots <i_{n_0+1}\leq d$ and $1\leq l\leq k$ we have
$$
\mathcal{W}\left(x\mapsto x^l \cdot\frac{d^l}{dx^l}\left(F_{i_1}\ast \dots \ast F_{i_{n_0+1}}\right)(x)\right).
$$
Without loss of generality we may assume that $i_s=s$ for $1\leq s\leq n_0+1$. We obtain
$$
F_1\ast \dots\ast F_{n_0+1}=(F_1\ast \dots \ast F_{n_0-1})\ast (F_{n_0}\ast F_{n_0+1}).
$$
To make the notation shorter, let $G_1:=F_1\ast \dots \ast F_{n_0-1}$ and $G_2:=F_{n_0}\ast F_{n_0+1}$. Then
\begin{multline}\label{eq:roww}
\left|\frac{d^l}{dx^l}(G_1\ast G_2) (x) \right|\leq \int_0^{x/2}\left| G_1(y)\frac{d^l}{dx^l}G_2(x-y)\right|\ dy\\
+ \int_{0}^{x/2} \left|G_2(y) \frac{d^l}{dx^l}G_1(x-y)\right|\ dy \\
+ \sum_{j=0}^{l-1}\overline{w}_j \left| \frac{d^j}{dx^j} G_1(x/2)\frac{d^{l-1-j}}{dx^{l-1-j}}G_2(x/2)\right|
\end{multline}
for some $\overline{w}_j$, $0\leq j\leq l-1$. Let $A>1$ and $t>0$ be such that we have
\begin{equation}\label{zalind}
\mathcal{W}\left(x\mapsto x^j\cdot\frac{d^j}{dx^j}G_1(x),A,t/2\right)
\end{equation}
and
\begin{equation}\label{1krok}
\mathcal{W}\left(\frac{d^j}{dx^j}G_2,A,t/2\right)\text{ for }1\leq j\leq k.
\end{equation}
By Lemma~\ref{lm:lemat5} it follows that without loss of generality we may assume that the above properties also hold for $j=0$.
Fix $1\leq l\leq k$.

We will show now that
$$
\mathcal{W}_1\left(x\mapsto x^l\cdot \frac{d^l}{dx^l}(G_1\ast G_2)(x)\right)
$$ 
holds. Take $x> t$. Using the inequality~\eqref{eq:roww} we will estimate now the value of $\left| x^l \cdot \frac{d^l}{dx^l}(G_1\ast G_2) (x) \right|$. For some $M>1$ we have
\begin{multline}\label{po:a}
x^l\cdot \int_0^{x/2}\left|G_1(y) \frac{d^k}{dx^k}G_2(x-y) \right|\ dy \leq Ax^l\cdot e^{-\frac{x}{A}} \int_{0}^{x/2} |G_1(y)|e^{\frac{y}{A}}\ dy\\
\leq Ax^l \cdot e^{-\frac{x}{2A}}\cdot\int_{0}^{x/2} |G_1(y)| \ dy < M e^{-\frac{x}{M}}.
\end{multline}
Moreover, since $y\in(0,x/2)$, we have $x-y>x/2$, whence for some $M>1$
\begin{multline}\label{po:b}
x^l\cdot \int_{0}^{x/2} \left |G_2(y) \frac{d^l}{dx^l}G_1(x-y)\right| \ dy\\
 < 2^l \cdot\int_{0}^{x/2} |G_2(y)| \cdot \left|(x-y)^l\cdot \frac{d^l}{dx^l}G_1(x-y) \right| \ dy\\
 <2^l\cdot A e^{-\frac{x}{2A}}\cdot \int_{0}^{x/2}|G_2(y)|\ dy < Me^{-\frac{x}{M}},
\end{multline}
where the one before the last inequality follows by the inductive assumption~\eqref{zalind}. Moreover, for $0\leq j\leq l-1$
\begin{multline}\label{po:c}
x^l\cdot \left|\frac{d^j}{dx^j}G_1(x/2) \frac{d^{l-1-j}}{dx^{l-1-j}}G_2(x/2) \right|\\
=x^{j+1}\cdot\left|\frac{d^j}{dx^j}G_1(x/2) \right|\cdot  \left| x^{l-1-j} \cdot \frac{d^{l-1-j}}{dx^{l-1-j}}G_2(x/2) \right|\\
<A^2x^{j+1}\cdot e^{-\frac{x}{A}}<Me^{-\frac{x}{M}}
\end{multline}
for some $M>1$ (this constant is independent of $x$). By~\eqref{po:a},~\eqref{po:b} and~\eqref{po:c} it follows that there exists $B>1$ such that for $x> t$ we have
\begin{equation*}
\left| x^l \cdot \frac{d^l}{dx^l}G_1\ast G_2 (x) \right|< Be^{-\frac{x}{B}}.
\end{equation*}
Therefore $\mathcal{W}_1\left(x\mapsto x^l\cdot \frac{d^l}{dx^l}(G_1\ast G_2)(x)\right)$ holds.

We will show now that
$$
\mathcal{W}_2\left(x\mapsto x^l\cdot \frac{d^l}{dx^l}(G_1\ast G_2)(x)\right)
$$
holds. Let $0<x<t$. Using again the inequality~\eqref{eq:roww}, we obtain an estimate for $\left| x^l \cdot \frac{d^l}{dx^l}(G_1\ast G_2) (x) \right|$. Since the function $G_1$ is integrable and the function $y\mapsto \frac{d^l}{dx^l}G_2(x-y)$ is bounded on the interval $(0,t/2)$, therefore the expression
\begin{equation*}
\int_0^{x/2}\left|G_1(y)\frac{d^l}{dx^l}G_2(x-y) \right|\ dy
\end{equation*}
is bounded, whence it follows that for some $M>1$ (independent of the choice of $x$)
\begin{equation}\label{eq:ttt}
x^l\cdot  \int_0^{x/2}\left|G_1(y)\frac{d^l}{dx^l}G_2(x-y) \right|\ dy < M.
\end{equation}
Moreover, using the inequality $x-y>x/2$ for $y\in (0,x/2)$, by boundedness of the function $G_2$ on the interval $(0,t/2)$ and by~\eqref{zalind} and Remark~\ref{uw:nowauwaga}, there exist constants $C,M>1$ (independent of $x$) such that the following inequalities hold:

\begin{multline}\label{eq:ttt2}
x^l\cdot \int_0^{x/2}\left|G_2(y) \frac{d^l}{dx^l}G_1(x-y) \right|\ dy \\
< 2^l\cdot \int_0^{x/2}\left|G_2(y) (x-y)^l\frac{d^l}{dx^l}G_1(x-y) \right|\ dy\\
<C \int_{0}^{x/2}\left|(x-y)^l \cdot \frac{d^l}{dx^l}G_1(x-y) \right| \ dy\\
= C\int_{x/2}^{x}\left|y^l \cdot\frac{d^l}{dx^l}G_1(y) \right|\ dy\\
<C\int_{0}^{t}\left|y^l \cdot\frac{d^l}{dx^l}G_1(y) \right|\ dy\\
=C\int_{0}^{t/2}\left|y^l \cdot\frac{d^l}{dx^l}G_1(y) \right|\ dy +C\int_{t/2}^{t}\left|y^l \cdot\frac{d^l}{dx^l}G_1(y) \right|\ dy\\
<M\int_{0}^{t/2}y^r \ dy + M,
\end{multline}
where the last expression (by the assumption we have $r>-1$) is also bounded. Moreover, for $0\leq j\leq l-1$, for some $M>1$ (independent of $x$), as in~\eqref{po:c} we obtain
\begin{equation}\label{eq:ttt3}
x^l \cdot\left|\frac{d^j}{dx^j}G_1(x/2) \frac{d^{l-1-j}}{dx^{l-1-j}}G_2(x/2) \right|<Me^{-\frac{x}{M}}.
\end{equation}
By~\eqref{eq:ttt},~\eqref{eq:ttt2} and~\eqref{eq:ttt3} it follows that there exists $B>1$ such that for $0<x<t$ we have
\begin{equation*}
\left| x^l\cdot \frac{d^l}{dx^l}G_1\ast G_2(x)\right|<Bx^r,
\end{equation*}
whence $\mathcal{W}_2\left(x\mapsto x^l\cdot \frac{d^l}{dx^l}(G_1\ast G_2)(x)\right)$ also holds.

We have shown that $\mathcal{W}\left(x\mapsto x^l\cdot \frac{d^l}{dx^l}(G_1\ast G_2)(x)\right)$ holds for $1\leq l\leq k$, which completes the proof of the inductive step and of the whole proposition.
\end{proof}

\begin{uw}\label{ponim}
Under the assumptions of Proposition~\ref{lm:wlasnoscPdlasplotu}, for $d\geq 1$ even, for any $k\geq 1$
\begin{equation*}
\mathcal{W}\left(\frac{d^k}{dx^k}(F_1\ast \dots \ast F_d) \right)
\end{equation*}
holds. The proof of this fact goes along the same lines as the proof of Proposition~\ref{lm:wlasnoscPdlasplotu}.
\end{uw}

\begin{pr}\label{lm:3.33}
Let $G_1\in L^1(0,\infty)$ and $G_2 \in L^1(-\infty,0)$ satisfy the assumptions of Proposition~\ref{stw:deriv}, i.e.
\begin{itemize}
\item
the functions $G_1$ and $G_2$ are smooth,
\item
the function $G_2$ is analytic at zero.
\end{itemize}
Moreover, assume that for any $k\in \N$ we have
\begin{itemize}
\item
$\mathcal{W}\left(x \mapsto x^k \cdot \frac{d^k}{dx^k}G_1(x)\right)$,
\item
$\mathcal{W}\left(x \mapsto  \frac{d^k}{dx^k}G_2(x)\right)$.
\end{itemize}
Then for any $k\in\N$ we have 
$$
\mathcal{W}\left(x \mapsto x^k\cdot \frac{d^k}{dx^k}(G_1\ast G_2)(x)\right).
$$
\end{pr}

\begin{proof}
We will use the formulas for the derivatives of a convolution from Proposition~\ref{stw:deriv}. Let $A>1$ and $t>0$ be such that for $0\leq j\leq k$ we have
$$
\mathcal{W}\left(x \mapsto x^j \cdot \frac{d^j}{dx^j}G_1(x),A,t\right)\text{ and }\mathcal{W}\left(x\mapsto \frac{d^j}{dx^j}G_2(x),A,t \right).
$$
By Proposition~\ref{stw:deriv} we obtain for $x>0$
\begin{multline*}
\left|\frac{d^k}{dx^k}(G_1\ast G_2)(x) \right|\\
\leq \int_x^{\infty}|G_1(y)|\cdot\left|\frac{d^k}{dx^k}G_2(x-y) \right|\ dy
+\sum_{l=0}^{k-1}\left|\frac{d^{k-l-1}}{dx^{k-l-1}}G_1(x) \right|\cdot \left|\frac{d^l}{dx^l}G_2(0) \right|,
\end{multline*}
whereas for $x<0$ we have
\begin{equation*}
\left|\frac{d^k}{dx^k}(G_1\ast G_2)(x) \right|\leq \int_0^{\infty}\left|G_1(y)\frac{d^k}{dx^k}G_2(x-y) \right| \ dy.
\end{equation*}

We will show now that
\begin{equation}\label{eq:trt}
\mathcal{W}_1\left(x \mapsto x^k\cdot \frac{d^k}{dx^k}(G_1\ast G_2)(x)\right)
\end{equation}
holds. Let $x> t$. For $y\in (x,\infty)$ we have $|G_1(y)|\leq A e^{-\frac{y}{A}}$. Moreover, the function $y\mapsto \frac{d^k}{dx^k}G_2(x-y)$ is bounded on the interval $(t,\infty)$. Therefore
\begin{multline*}
x^k\cdot \int_x^{\infty}|G_1(y)|\cdot\left|\frac{d^k}{dx^k}G_2(x-y) \right|\ dy  \\
\leq Cx^k\cdot \int_x^{\infty}Ae^{-\frac{y}{A}}\ dy=Cx^k\cdot A^2\cdot e^{-\frac{x}{A}} < Me^{-\frac{x}{M}}
\end{multline*} 
for some $C,M>1$ (these constants are independent of $x$). Moreover, for $0\leq l\leq k-1$
\begin{equation*}
x^{k-l-1}\cdot \left| \frac{d^{k-l-1}}{dx^{k-l-1}}G_1(x)\right|<Ae^{-\frac{x}{A}}.
\end{equation*}
It follows that
\begin{equation}\label{eq:ww1}
x^k\cdot \left|\frac{d^k}{dx^k} (G_1\ast G_2)(x)\right|<Be^{-\frac{x}{B}}
\end{equation}
for some $B>0$ for $x>t$.

Let $x< -t$. We have
\begin{multline*}
\left|\frac{d^k}{dx^k}(G_1\ast G_2)(x) \right|\\
\leq \int_0^t |G_1(y)|\left|\frac{d^k}{dx^k}G_2(x-y) \right|\ dy + \int_t^{\infty}|G_1(y)| \left|\frac{d^k}{dx^k}G_2(x-y) \right|\ dy.
\end{multline*}
For $y>t$ we have $|G_1(y)|<Ae^{-\frac{y}{A}}$ and the function $y\mapsto \frac{d^k}{dx^k}G_2(x-y)$ is bounded (this bound is independent of $x$), whence
$$
\int_t^{\infty}|G_1(y)| \left|\frac{d^k}{dx^k}G_2(x-y) \right|\ dy \leq  B e^{-\frac{x}{B}}
$$
for some $B>1$. We will estimate now $\int_0^t |G_1(y)|\left|\frac{d^k}{dx^k}G_2(x-y) \right|\ dy$. The function $y\mapsto y^{-r}G_1(y)$ is bounded on $(0,t)$. Since the function $G_2$ is analytic at zero and we have $\mathcal{W}\left(\frac{d^k}{dx^k}(G_2)\right)$, therefore there exists $B>1$ such that
$$
\left|\frac{d^k}{dx^k}G_2(y) \right|<B e^{-\frac{|y|}{B}}\text{ for any }y<0.
$$
Therefore, there exist $M>1$ and $C>1$ such that
\begin{multline*}
\int_0^t |G_1(y)|\left|\frac{d^k}{dx^k}G_2(x-y) \right|\ dy \leq M\int_0^t y^r \cdot e^{-\frac{|x-y|}{M}}\ dy\\
\leq Me^{-\frac{|x|}{M}}  \cdot  \int_0^t y^r\ dy \leq C e^{-\frac{|x|}{C}}.
\end{multline*}
Hence, for some $D>1$ and $x\leq-t$
\begin{equation}\label{eq:ww}
\left|x^k\cdot \frac{d^k}{dx^k}(G_1\ast G_2)(x) \right|<D e^{-\frac{|x|}{D}}.
\end{equation}
By~\eqref{eq:ww1} and~\eqref{eq:ww} we deduce that~\eqref{eq:trt} indeed holds.

We will show now that
\begin{equation}\label{eq:trt2}
\mathcal{W}_2\left(x \mapsto x^k\cdot \frac{d^k}{dx^k}(G_1\ast G_2)(x)\right)
\end{equation}
holds. Let $x\in (0,t)$. Notice that the function $y\mapsto \frac{d^k}{dx^k}G_2(x-y)$ is bounded on $(x,\infty)$ (this bound is independent of $x$). Therefore and by the integrability of $G_1$ we obtain
$$
\int_x^{\infty}|G_1(y)| \cdot\left|\frac{d^k}{dx^k}G_2(x-y) \right| \ dy<M
$$
for some $M>1$. Moreover, for $0\leq l\leq k-1$ the expression 
$$
x^{-r}\cdot x^{k-l-1}\cdot \left| \frac{d^{k-l-1}}{dx^{k-l-1}}G_1(x)\right|
$$
is bounded (again, this bound is independent of $x$). Therefore there exists $B>1$ such that for $0<x<t$
\begin{equation}\label{eq:pp}
\left|x^k\cdot \frac{d^k}{dx^k}(G_1\ast G_2)(x) \right| <Mx^r
\end{equation}
for some $M>1$.

Let $x \in (-t,0)$. Notice that the function $y\mapsto \frac{d^k}{dx^k}G_2(x-y)$ is bounded, whence by the integrability of the function $G_1$, there exists $M>1$ such that
\begin{equation}\label{eq:pp2}
|x|^k\cdot \int_0^{\infty}|G_1(y)|\cdot \left|\frac{d^k}{dx^k}G_2(x-y) \right|\ dy <M.
\end{equation}

By~\eqref{eq:pp} and~\eqref{eq:pp2} we obtain that~\eqref{eq:trt2} indeed holds, which ends the proof.
\end{proof}

\subsubsection{Proof of the main technical result}\label{se:techni}

\begin{proof}[Proof of Proposition~\ref{tw:nowametoda}]
We will show first that without loss of generality we may assume that $a=0$. Indeed, let $d\geq 1$ and take $t_1,\dots,t_d,t_1',\dots,t_d'\in\R\setminus\{0\}$ such that  $t_1+\dots+t_d=t_1'+\dots+t_d'$. For $1\leq i\leq d$ let $x_i=1/t_i$ and $x_i'=1/t_i'$ and let $\oh(x)=h(x+a)$. Then for $t\neq 0$ we have $\oh_t(x)=h_t(x+ta)$. We claim that for $t_1,\dots,t_d\neq 0$
\begin{equation}\label{poru}
\oh_{t_1}\ast \dots \ast \oh_{t_d}(x)=h_{t_1}\ast \dots \ast h_{t_d}(x+(t_1+\dots +t_d)\cdot a).
\end{equation}
We will show that this formula holds for $d=2$ (one uses induction to obtain this result for larger $d$). We have
\begin{multline*}
\oh_{t_1}\ast \oh_{t_2}(x)=\int_{\R}\oh_{t_1}(y)\oh_{t_2}(x-y)\ dy\\
=\int_{\R}h_{t_1}(y+t_1a)h_{t_2}(x-y+t_2a)\ dy=h_{t_1}\ast h_{t_2}(x+(t_1+t_2)a).
\end{multline*}
It follows by~\eqref{poru} that
$$
h_{t_1}\ast \dots \ast h_{t_d}=h_{t_1'}\ast \dots\ast  h_{t_d'} \Leftrightarrow \oh_{t_1}\ast \dots\ast \oh_{t_d}=\oh_{t_1'}\ast \dots \ast\oh_{t_d'}
$$
and instead of the density $h$ of the measure $P$ we may consider its translation $\oh$. 

Without loss of generality we may assume that $1\leq i_0 \leq d$ is such that
\begin{equation*}
t_1\leq \dots \leq t_{i_0} <0 <t_{i_0+1} \leq \dots \leq t_d.
\end{equation*}
Since $d$ is even, the cardinalities of the sets $\{t_1,\dots,t_{i_0}\}$ and $\{t_{i_0+1},\dots, t_d\}$ are of different parity. To fix attention, assume that $i_0$ is even. Then by Proposition~\ref{lm:wlasnoscPdlasplotu} and by Remark~\ref{ponim}, for any $k\geq 1$ the following properties hold:
\begin{equation*}
\mathcal{W}\left(x\mapsto x^k\cdot \frac{d^k}{dx^k}(h_{t_{i_0+1}}\ast \dots\ast h_{t_d})(x)\right)
\end{equation*}
and
\begin{equation*}
\mathcal{W}\left(x\mapsto \frac{d^k}{dx^k}(h_{t_1}\ast \dots\ast h_{t_{i_0}})(x)\right).
\end{equation*}
Therefore, by Proposition~\ref{lm:3.33}, for any $k\geq 1$ we have
\begin{equation*}
\mathcal{W}\left(x\mapsto x^k\cdot \frac{d^k}{dx^k}(h_{t_1}\ast \dots \ast h_{t_d})(x)\right).
\end{equation*}

By Lemma~\ref{stw:jjj} we conclude further that all the derivatives of the function $h_{t_1}\ast \dots \ast h_{t_d}$ decay ``sufficiently fast''. Similarly, for any subset $\{s_1,\dots,s_k\}\subset \{t_1,\dots,t_{i_0}\}$ or $\{s_1,\dots,s_k\}\subset \{t_{i_0+1},\dots,t_d\}$ all the derivatives of the function $h_{s_1}\ast \dots \ast h_{s_k}$ decay ``sufficiently fast''. Therefore by Corollary~\ref{wn:wielomianytu} we obtain that the numbers
$$
B_n=\sum_{k_1+\dots+k_d=n} \prod_{i=1}^{d} \frac{a_{k_i}\Gamma(k_i+1/2)}{t_i^{k_i}}
$$
are uniquely determined. Hence also the set $\{t_1,\dots,t_d\}$ is uniquely determined.

Notice that by the continuity of the measure $\sigma$ it follows that 
\begin{equation*}
\sigma^{\otimes d}\left( \{(t_1,\dots,t_d)\in\R^d\colon t_i\neq 0 \text{ for }1\leq i\leq d\} \right)=1.
\end{equation*}
To end the proof it suffices to apply Proposition~\ref{pr:l+p}.
\end{proof}

\subsubsection{Symmetric polynomials}\label{se:ws}
In Proposition~\ref{tw:nowametoda} an important role is played by some symmetric polynomials which are related to the density of the measure $P$ appearing in the integral operator in the so-called weak closure of times. A crucial assumption is that the values of these polynomials determine the values of the variables (up to a permutation of their names). We will now consider the following related problem. Given $d\geq 1$ and a sequence of reals $(b_k)_{k\in\N}$ we consider a system of an infinite number of equations ($n\geq 0$):
\begin{equation}\label{eq:uklrow}
\sum_{k_i\in \N,k_1+\dots+k_d=n}\prod_{i=1}^{d}b_{k_i}x_i^{k_i}=c_n.
\end{equation}
We will provide a partial answer to the following question:

\begin{equation}\label{Q:jpt} 
\begin{array}{l} 
\mbox{When does the fact that the above system of equations has a solution}\\
\mbox{imply that the solution is unique}\\
\mbox{(up to a permutation of the names of the variables)?}
\end{array}
\end{equation}

Let us introduce first the necessary notation. For natural numbers $\alpha_1,\alpha_2,\dots , \alpha_d$ such that $\alpha_1\geq \alpha_2\geq \dots \geq \alpha_d\geq 0$ we put
$$
m_{\alpha_{1},\dots,\alpha_{d}}=\sum_{\pi\in S}x_1^{\alpha_{\pi(1)}}\cdot\ldots \cdot x_d^{\alpha_{\pi(d)}},
$$
where $S$ is the maximal subset of the set of permutations of  $\{1,\dots,d\}$ such that for $\pi_1,\pi_2\in S,\ \pi_1\neq\pi_2$ we have 
$$(\alpha_{\pi_1(1)},\dots,\alpha_{\pi_1(d)})\neq (\alpha_{\pi_2(1)},\dots,\alpha_{\pi_2(d)}).$$
If there are some zeros among the numbers $\alpha_i$ we will skip them and write e.g. $m_{2,1}$ instead of $m_{2,1,0,0}$.

\begin{uw}
The polynomials $m_1,m_{1,1},\dots,m_{\underbrace{1,\dots,1}_{d}}$ are called \emph{elementary symmetric polynomials} and (by the fundamental theorem of algebra) their values determine the values of the variables $x_1,\dots,x_d$ up to a permutation of their names.
\end{uw}

We claim that under some assumptions on the numbers $b_n$ ($n\in\N$) the system of equations~\eqref{eq:uklrow} has not more than one solution (up to a permutation of the names of the variables). To make the arguments easier to follow, we will show a few first steps, in which we will recover the information about the polynomials $m_{\underbrace{1,\dots,1}_{n}}$ for consecutive natural numbers $n$ from the system of equations. Simultaneously we will state successive conditions for the sequence $(b_n)_{n\in\N}$ sufficient for our inductive procedure so that it can be continued as long as necessary.

\noindent
\textbf{Step 1.}
\\
The equation~\eqref{eq:uklrow} for $n=1$ gives the information about the value of the polynomial $m_1$ provided that $b_0,b_1\neq 0$.
\\
\textbf{Step 2.}
\\
The equation~\eqref{eq:uklrow} for $n=2$ is of the form
\begin{equation*}
b_1^2\cdot b_0^{d-2}m_{1,1}+b_2b_0^{d-1}m_2=c_2.
\end{equation*}
Notice that the following identity holds:
\begin{equation*}
m_{1,1}+m_2=m_1^2.
\end{equation*}
We treat now the polynomials $m_{1,1}$ and $m_2$ as unknowns. Since we already know $m_1$, we have obtained this way a system of two linear equations for these unknowns. For the uniqueness of the solution it is necessary and sufficient that these equations are linearly independents. This means that it suffices that $b_2 \neq C_2(b_0,b_1)$, where $C_2$ is some function of two variables. Using the system of equation which we have just derived from the original one, one can deduce the precise formula for this function:
$$C_2(b_0,b_1)=\frac{b_1^2}{b_0}.$$
\\
\textbf{Step 3.}
\\
The equation~\eqref{eq:uklrow} for $n=3$ is of the form
\begin{equation*}
b_1^3b_0^{d-3}m_{1,1,1}+b_2b_1b_0^{d-2}m_{2,1}+b_3b_0^{d-1}m_3=c_3.
\end{equation*}
Notice that the following identities hold:
\begin{align*}
6m_{1,1,1}+3m_{2,1}+m_3=&m_1^3,\\
m_{2,1}+m_3=m&_2m_1.
\end{align*}
We treat now the polynomials $m_{1,1,1}, m_{2,1}$ and $m_3$ as unknowns.
Since we already know $m_1$ and $m_2$, we have obtained this way a system of three linear equations for these unknowns:
\begin{equation*}
\left\{
\begin{array}{rrrl}
6\cdot m_{1,1,1}+&3\cdot m_{2,1}+&m_3&=m_1^3\\
 &m_{2,1}+&m_3&=m_2m_1\\
 b_1^3b_0^{d-3}\cdot m_{1,1,1}+&b_2b_1b_0^{d-2}\cdot m_{2,1}+&b_3b_0^{d-1}\cdot m_3&=c_3.
\end{array} 
\right.
\end{equation*}
For the uniqueness of the solution it is necessary and sufficient that these equations are linearly independent. This means that it suffices that $b_3 \neq C_3(b_0,b_1,b_2)$, where $C_3$ is some function of three variables. To determine the precise formula, we check that the determinant of the matrix
\begin{equation*}
\left(
\begin{array}{ccc}
6 & 3 & 1\\
0 & 1 & 1\\
b_1^3 & b_2b_1b_0 & b_3b_0^2
\end{array}
\right)
\end{equation*}
is non-zero if and only if $b_3\neq C_3(b_0,b_1,b_2)$, where
$$C_3(b_0,b_1,b_2)=\frac{3b_0b_1b_2-b_1^3}{3b_0^2}.$$

We claim that this procedure can be continued. It is not a coincidence that each time the number of equations and the number of unknowns were the same. Both the unknowns (in Step 3. we had $m_{1,1,1},m_{2,1},m_{3}$) and the right-hand sides of the identities (in Step 3. we had  $m_1^3=m_1m_1m_1$ and $m_2m_1$) are determined by partitions of $n$\footnote{We say that a decomposition of the natural number $n$ into a sum of natural numbers $n=n_1+n_2+\dots+n_k$ is called a \emph{partition} when $n_1\geq n_2\geq \dots \geq n_k$.} (in Step 3. $n=3$). Since the trivial partition $n=n$ yields the identity $m_n=m_n$, which does not convey any information, we replace it with the equation coming from the system~\eqref{eq:uklrow}. This means that at each step we have indeed the same number of equations and unknowns. Moreover, the equation given by the derived identities are linearly independent (if we hadn't replaced the equation $m_n=m_n$ with the equation coming from the system~\eqref{eq:uklrow}, the coefficient matrix would be upper-triangular). Whether the equation coming from the system~\eqref{eq:uklrow} is linearly independent of the other equations depends on whether  $b_n\neq C_n(b_0,b_1,\dots,b_{n-1})$, where $C_n$ is some function of $n$ variables, whose form can be determined at each step from the derived identities. If for $1\leq n\leq d$ we have $b_n\neq C_n(b_0,b_1,\dots,b_{n-1})$, then after performing $d$ steps, we know the values of the polynomials $m_1,m_{1,1},\dots,m_{{\underbrace{1,\dots,1}_{d}}}$. Therefore the solution is unique, up to a permutation of the names of the variables.

Unfortunately, in a concrete situation, it may turn out that it is a very difficult task to check if at all steps of the inductive procedure we have obtained a system of linearly independent equations. If at some step this is not true, one needs other methods. The described procedure clearly doesn't exhaust all the cases when given a system of equations of the form~\eqref{eq:uklrow}, we can ``read'' the values of the elementary symmetric polynomials which yield the values of all the variables (up to a permutation of their names). With such a situation we will deal in the next section. The numbers $(b_n)_{n\in\N}$ in our example will be such that some of the equations~\eqref{eq:uklrow} will not carry any additional information, i.e. the system of equation obtained using the described procedure will be not linearly independent.

\subsection{Application}\label{se:wyniki2}
The main goal in this section is to apply the results from Section~\ref{se:wyniki1} to the class of flows described in Section~\ref{se:potoki}. 

Recall that $\cT=(T_t)_{t\in\R}$ stands for the special flow over an irrational rotation on the circle $Tx=x+\alpha\ (\!\!\!\!\mod 1)$, $\alpha\in (0,1)\cap \R\setminus \Q$ under a roof function of the form $f+f_1+c\colon [0,1)\to \R$, where
$$
f(x)=-\ln(x)-\ln(1-x)-2,
$$
$f_1\colon\T\to \R$ is an absolutely continuous function with zero average and $c\in\R$ is such that $f+f_1+c>0$. Assume additionally that
$$
\liminf_{n \to \infty} q_n^3 \||q_n \alpha\| =0,
$$
where $q_n$ are the denominators in the continued fraction expansion of $\alpha$.

Let us now state the main result of this section.

\begin{tw}\label{tw:5.19}
The maximal spectral type $\sigma$ of the flow $\cT$ 
is such that the unitary flow $V_\sigma^{\odot 3}$ has simple spectrum.
\end{tw}

Recall one more result on the simplicity of spectra of unitary representations.\footnote{It was stated in~\cite{1112.5545} for automorphisms, i.e. $\mathbb{Z}$-actions but it is valid also for $\R$-representations.} 
\begin{lm}\cite{1112.5545}
If for some $k\geq 1$ the unitary operator $U^{\odot k}$ has simple spectrum then also the operators $U^{\odot j}$ for $1\leq j\leq k-1$ have simple spectra.
\end{lm}

By the above lemma we obtain the following.
\begin{wn}
The maximal spectral type $\sigma$ of the flow $\cT$ is such that the flow $V_\sigma^{\odot 2}$ has simple spectrum.
\end{wn}

By Theorem~\ref{tw:5.19} and by the remarks included in Section~\ref{se:potoki} we obtain immediately the following corollary.
\begin{wn}
On any closed orientable surface of genus at least $2$ there exists a smooth flow whose maximal spectral type $\sigma$ is such that the unitary flow $V_{\sigma}^{\odot 3}$ has simple spectrum.
\end{wn}

\begin{uw}
The question whether the maximal spectral type of the considered flows has the SCS property remains open. We think that the answer is positive. We believe that using computer analysis one could show that $V_\sigma^{\odot n}$ has simple spectrum also for some larger $n> 3$.
\end{uw}

We will state now some auxiliary results needed for the proof of Theorem~\ref{tw:5.19}, which we believe that can be also of an independent interest.

\subsubsection{Limit distributions}
We will concentrate now on finding distributions satisfying condition~\eqref{eq:zbieznosc} (see page~\pageref{eq:zbieznosc}) for $\cT$. Let $(q_n)_{n\in\N}$ be the sequence of the denominators and $(p_n)_{n\in\N}$ the sequence of the numerators of $\alpha$ in the continued fraction expansion (see Section~\ref{se:cf}).

By the classical Denjoy-Koksma inequality (see e.g.~\cite{MR832433}) $f_1^{(q_n)} \to 0$ uniformly. Similarly, also $f_1^{(mq_n)}\to 0$ uniformly for any $m\in\Z$. Therefore, while calculating the limit distribution of the sequence of functions $(f+f_1)^{(q_n)}$, we may skip the function $f_1$ and calculate instead the limit distribution of the sequence of functions $f^{(q_n)}$.

For $q\in\N$ by $f_q \colon (0,1)\to\R$ let us denote the following function:
\begin{equation}\label{eq:f_q}
f_q(x)=\sum_{k=0}^{q-1}f\left(x+\frac{k}{q}\right).
\end{equation}
It is a periodic function with period $\frac{1}{q}$, which has the same distribution as the function $\tilde{f}_q \colon (0,1) \to \R$ given by the formula
\begin{equation}\label{eq:falka}
\tilde{f}_q(x)=f_q\left(\frac{x}{q}\right).
\end{equation}
\begin{lm}\label{lm:zbieznosc}
The sequence of functions $\left(\widetilde{f}_q\right)_{q\in\N}$ converges uniformly to $\ln\frac{1}{2\sin(\pi \cdot)}$ on the interval $(0,1)$ as $q$ tends to $\infty$.
\end{lm}
\begin{proof}
We have
\begin{align*}
\tilde{f}_q(x)&=\sum_{k=0}^{q-1}f\left(\frac{x}{q}+\frac{k}{q} \right)=-2q-\sum_{k=0}^{q-1} \ln\left(\frac{x+k}{q} \right)-\sum_{k=0}^{q-1}\ln \left(\frac{q-k-x}{q} \right)=\\
&=-2q+2q\ln(q)-\ln\left( \prod_{k=0}^{q-1}(x+k) \prod_{k=1}^{q}(-x+k) \right)=\\
&=-2q+2q\ln(q)-\ln\left( \frac{\Gamma(x+q)}{\Gamma(x)}\frac{\Gamma(-x+q+1)}{\Gamma(1-x)} \right).
\end{align*}
Since $\Gamma(x)\Gamma(1-x)=\frac{\pi}{\sin(\pi x)}$ (see e.g.~\cite{abramowitz+stegun}),
we obtain
\begin{equation*}
\tilde{f}_q(x)=-2q+2q\ln(q)-\ln\left(\Gamma(q+x)\Gamma(q+1-x)\right)+\ln\left(\frac{\pi}{\sin(\pi x)}\right).
\end{equation*}
We claim that
\begin{equation}\label{eq:uniform}
\lim_{q\to\infty}r_q(x)=0,
\end{equation}
where
\begin{equation*}
r_q(x)=-2q+2q\ln(q)-\ln(\Gamma(q+x)\Gamma(q+1-x))+\ln2\pi
\end{equation*}
and the convergence is uniform on the interval $(0,1)$. To this end we will show that
\begin{equation}\label{eq:pomocnicza}
\Gamma\left(q+\frac{1}{2}\right)\Gamma\left(q+\frac{1}{2}\right)\leq \Gamma(x+q)\Gamma(q+1-x)\leq \Gamma(q)\Gamma(q+1).
\end{equation}
For $0\leq y \leq \frac{1}{2}$ let
\begin{equation*}
G(y)=\Gamma(q+\frac{1}{2}+y)\Gamma(q+\frac{1}{2}-y).
\end{equation*}
We have
\begin{multline*}
G'(y)=\Gamma(q+\frac{1}{2}+y)\Gamma'(q+\frac{1}{2}-y)+\Gamma'(q+\frac{1}{2}+y)\Gamma(q+\frac{1}{2}-y)\\
=\Gamma(q+\frac{1}{2}+y)\Gamma(q+\frac{1}{2}-y)\left(\Psi(q+\frac{1}{2}+y)-\Psi(q+\frac{1}{2}-y) \right),
\end{multline*}
where $\Psi(x)=\frac{\Gamma'(x)}{\Gamma(x)}$ is the so-called function digamma. Since (see~\cite{Alzer:1997:SIG}, H.~Alzer)
\begin{equation*}
\Psi'(x)=\sum_{k=0}^{\infty}\frac{1}{(x+k)^2},
\end{equation*}
the function $\Psi$ is increasing, $G'(y)>0$ for $y>0$ and the function $G$ is also increasing. Therefore
\begin{equation*}
\Gamma \left(q+\frac{1}{2}\right) \Gamma\left(q+\frac{1}{2}\right)=G(0)\leq G(y)\leq G\left(\frac{1}{2} \right)=\Gamma(q)\Gamma(q+1),
\end{equation*}
whence~\eqref{eq:pomocnicza} holds. We will find now an upper and lower bound for $r_q(x)$ which are independent of $x$ and converge to $0$ when $q$ tends do $\infty$. We will first find an upper bound. By the left inequality in~\eqref{eq:pomocnicza}, by the equality
\begin{equation*}
\Gamma\left(z+\frac{1}{2}\right)=\frac{2^{1-2z}\sqrt{\pi}\cdot\Gamma(2z)}{\Gamma(z)}
\end{equation*}
(see e.g.~\cite{abramowitz+stegun}) and by the Stirling's formula\footnote{Recall the Stirling's formula: $\lim_{n\to \infty}\frac{n!}{\sqrt{2\pi n}\left(n/e\right)^n}$=1.} we obtain
\begin{align*}
r_q(x) &\leq \ln \left(e^{-2q}q^{2q}\frac{1}{\Gamma(q+\frac{1}{2})\Gamma(q+\frac{1}{2})}\cdot2\pi \right)\\
&=\ln \left(2\pi \left(e^{-q}q^q \frac{\Gamma(q)}{2^{1-2q}\sqrt{\pi}\cdot\Gamma(2q)}  \right)^2 \right)\\
&=\ln\left(2\pi \left( e^{-q}\cdot q^q \cdot \frac{1}{2^{1-2q}}\cdot\frac{1}{\sqrt{\pi}}\cdot\frac{2q}{q}\cdot\frac{q!}{(2q)!}\right)^2 \right)\\
&=\ln\left(2\pi \left( e^{-q}\cdot q^q \cdot \frac{1}{2^{1-2q}}\cdot\frac{1}{\sqrt{\pi}}\cdot\frac{2q}{q}\cdot\frac{q!}{\sqrt{2\pi q}\cdot\frac{q^q}{e^q}}\cdot\sqrt{2\pi q}\cdot\frac{q^q}{e^q}  \right.\right.\\
&\left.\left. \cdot\frac{\sqrt{4\pi q}\cdot\frac{(2q)^{2q}}{e^{2q}}}{(2q)!}\cdot \frac{1}{\sqrt{4\pi q}\cdot\frac{(2q)^{2q}}{e^{2q}}}   \right)^2 \right)\\
&=\ln\left(\left(\frac{q!}{\sqrt{2\pi q}\cdot\frac{q^q}{e^q}} \cdot \frac{\sqrt{4\pi q}\cdot\frac{(2q)^{2q}}{e^{2q}}}{(2q)!}\right)^2 \right) \to \ln 1=0 \text{ for }q\to\infty.
\end{align*}
In a similar way, by the right inequality in~\eqref{eq:pomocnicza} and the Stirling's formula we have
\begin{align*}
r_q(x)&\geq \ln\left(e^{-2q}q^{2q}\frac{1}{\Gamma(q)\Gamma(q+1)}\cdot2\pi \right)\\
&= \ln \left(e^{-2q}q^{2q}\cdot q\cdot(q+1)\cdot \frac{1}{q! \cdot (q+1)!}\cdot 2\pi \right)\\
&= \ln \left(e\cdot \left(\frac{q}{q+1}\right)^{q+1} \cdot \frac{q+1}{\sqrt{q(q+1)}} \cdot \frac{\sqrt{2\pi q} \cdot\frac{q^q}{e^q}}{q!} \cdot \frac{\sqrt{2\pi(q+1)}\cdot\frac{(q+1)^q}{e^{q+1}}}{(q+1)!} \right)\\
&\to \ln(e \cdot e^{-1} \cdot 1 \cdot 1)=\ln 1=0\text{ for }q\to\infty.
\end{align*}
Thus we have shown that $r_q(\cdot)$ converges to $0$ uniformly on the interval $(0,1)$, which by the definition of $\tilde{f}_q$ and $r_q$ completes the proof of the lemma.
\end{proof}

The next lemma together with Lemma~\ref{lm:zbieznosc} and Proposition~\ref{pr:l+p}

will let us find distributions $P$ satisfying the condition~\eqref{eq:zbieznosc} for some $t_n\to\infty$. Let 
\begin{equation}\label{eq:miaranu}
\nu=\left(\ln\left(\frac{1}{2\sin \pi \cdot} \right)\right)_{\ast}(\lambda)
\end{equation}
and let $M_t \colon \R \to \R$ be given by $M_t(x)=tx$ for $t\in\R \setminus \{0\}$. For $t\neq 0$ let 
\begin{equation}\label{eq:nu}
\nu_t=(M_t)_{\ast}(\nu).
\end{equation}

\begin{lm}\label{lm:lemma5}
Let $m\in \Z\setminus \{0\}$ and $\alpha\in\R\setminus\Q$ be such that

\begin{equation}\label{eq:DC}
\liminf_{n \to \infty} q_n^3 \||q_n \alpha\| =0.
\end{equation}
Then for some sequence $n_k\to \infty$ we have
\begin{equation*}
\left(f^{(mq_{n_k})}\right)_{\ast}\lambda \to \nu_m.
\end{equation*}
\end{lm}

\begin{proof}
Let $\alpha\in\R\setminus\Q$ fulfill the condition~\eqref{eq:DC} and let $m\in \N$. Let $\hat{f}(x)=-\ln x$ and let the sequence $\left(q_{n_k}\right)_{k\in\N}$ be such that
\begin{equation}\label{eq:podciag}
\lim_{k\to\infty}q_{n_k}^3 \|q_{n_k}\alpha\|=0.
\end{equation}
Without loss of generality we may assume that $\frac{p_{n_k}}{q_{n_k}}<\alpha$, i.e. $\|q_{n_k}\alpha\|=\left\{q_{n_k}\alpha\right\}$. To make the proof more readable, we will write $q$ instead of $q_{n_k}$. We claim that for any $k \in \Z$
\begin{equation*}
\left\|\hat{f}^{(q)}(x+k\|q\alpha\|)-\hat{f}_q(x))\right\|_{L^1} \to 0
\end{equation*}
(we define the function $\hat{f}_q$ in the same way as $f_q$ in~\eqref{eq:f_q} was defined). We fix $k\in\Z$. Let $n_0\in\N$ be sufficiently large, so that
\begin{equation}\label{eq:blisko}
\left\|q_{n_0}\alpha \right\|<\frac{1}{q_{n_0}^3} \text{ and }\frac{|k|+1}{q_{n_0}^3}<\frac{1}{2q_{n_0}}.
\end{equation}
For $q=q_n$ such that $n>n_0$ let 
\begin{equation*}
A_q= \bigcup_{j=0}^{q-1} \left[\frac{j}{q}+\frac{1}{q^3}+|k|\left\|q\alpha\right\|, \frac{j+1}{q}-\frac{1}{q^3}-|k|\left\|q\alpha \right\| \right].
\end{equation*}
Take $x\in A_q$. By the definition of the set $A_q$,  $x=\frac{j_0+t}{q}$ for some $0\leq j_0 \leq q-1$ and $t\in \left[\frac{1}{q^2}+|k|q\|q\alpha\|,1-\frac{1}{q^2}-|k|q\|q\alpha\| \right]$. Notice that each of the intervals $\left[\frac{j}{q},\frac{j+1}{q}\right)$ for $0\leq j\leq q-1$ contains exactly one point of the form $\{x+k\left\|q\alpha\right\|+k_j \alpha\}$, where $0\leq k_j\leq q-1$. Indeed, let $k_j$ be the only natural number in $[0, q-1]$ such that
\begin{equation}~\label{eq:where}
\left\{x+k_j \frac{p}{q}\right\} \in \left[\frac{j}{q}+\frac{1}{q^3}+|k|\left\|q\alpha\right\|,\frac{j+1}{q}-\frac{1}{q^3}-|k|\left\|q\alpha \right\| \right).
\end{equation}
We claim that $\left\{x+k\|q\alpha\|+k_j\alpha\right\} \in \left[\frac{j}{q},\frac{j+1}{q} \right)$. By~\eqref{eq:where} and 
\begin{multline*}
\left| k_j \left(\alpha-\frac{p}{q} \right)+k \|q\alpha \| \right| \\
\leq q\left| \alpha-\frac{p}{q}\right|+ |k|\|q\alpha\|= \|q\alpha \|+|k|\|q\alpha\|\leq \frac{1}{q^3}+|k|\|q\alpha\|,
\end{multline*}
we have
\begin{multline*}
\left\{x+ k\|q\alpha\|+k_j \alpha \right\}= \left\{x+k_j\frac{p}{q}+k_j \left(\alpha-\frac{p}{q} \right)+k\|q\alpha \| \right\}\\
=\left\{x+ k_j\frac{p}{q} \right\}+k_j \left(\alpha-\frac{p}{q} \right)+k\|q\alpha\| \in \left[\frac{j}{q},\frac{j+1}{q} \right).
\end{multline*}
The above calculations show also that
\begin{multline*}
\left\{x+k\|q\alpha\|+k_j\alpha \right\}=\frac{j}{q}+\frac{t}{q}+k_j\left(\alpha-\frac{p}{q} \right)+ k\|q\alpha\|\\
=\frac{j+t+k_j\left(q\alpha-p \right)+kq\|q\alpha\|}{q}=\frac{j+t+k_j\|q\alpha\|+kq\|q\alpha\|}{q}
\end{multline*}
Now, using the uniqueness of $k_j\in [0,q-1]$ in the above arguments, we can estimate $\left|\hat{f}^{(q)}(x+k\|q\alpha\|)-\hat{f}_q(x) \right|$ for $x \in A_q$. If $k \geq 0$,  we have
\begin{align*}
\left|\hat{f}^{(q)}(x+k\|q\alpha\|)-\hat{f}_q(x) \right|&=\sum_{j=0}^{q-1}\left(\ln\frac{j+t+k_j\|q\alpha\|+kq\|q\alpha\| }{q} -\ln\frac{j+t}{q}\right)\\
&\leq \sum_{j=0}^{q-1}\ln\left(1+\frac{(k+1)q\|q\alpha\|}{j+t} \right)\\
&\leq (k+1)q\|q\alpha\|\sum_{j=0}^{q-1}\frac{1}{j+t}\\
&\leq (k+1)q\|q\alpha\|\left(\frac{2}{t}+\ln q \right)\\
&\leq (k+1)q\|q\alpha\|(2q^2+\ln q) \to 0,
\end{align*}
where the last inequality follows by the fact that $t\geq \frac{1}{q^2}$, and the convergence is a direct consequence of~\eqref{eq:podciag}. If $k<0$, we proceed in a similar way. We have $k_j+kq<0$ for $0\leq k_j \leq q-1$. Moreover, if $q$ is sufficiently large then
\begin{equation*}
(k_j+kq)\|q\alpha\| \geq kq\|q\alpha \|=-|k|q\|q\alpha\|>-\frac{1}{2} \left(\frac{1}{q^2}+|k|q\|q\alpha\| \right) \geq -\frac{t}{2},
\end{equation*}
where the last inequality holds for all $t\in \left[\frac{1}{q^2}+|k|q\|q\alpha\|,1-\frac{1}{q^2}-|k|q\|q\alpha\| \right]$.
Therefore for any such $q$
\begin{align*}
\left|\hat{f}^{(q)}(x+k\|q\alpha\|)-\hat{f}_q(x) \right|&=\sum_{j=0}^{q-1}\left(\ln\frac{j+t}{q}-\ln\frac{j+t+k_j\|q\alpha\|+kq\|q\alpha\| }{q}\right)\\
&=\sum_{j=0}^{q-1}\ln \left(1+\frac{-(k_j+kq)\|q\alpha\|}{j+t+k_j\|q\alpha\|+kq\|q\alpha\|} \right)\\
&\leq \sum_{j=0}^{q-1}\ln \left(1+\frac{-(k_j+kq)\|q\alpha\|}{j+\frac{t}{2}} \right)\\
&\leq\sum_{j=0}^{q-1} -(k_j+kq)\|q\alpha\|\frac{1}{j+\frac{t}{2}}\leq \sum_{j=0}^{q-1} |k|q\|q\alpha\|\frac{1}{j+\frac{t}{2}}\\
&\leq |k|q\|q\alpha\|\left(\frac{4}{t}+\ln q \right)\\
&\leq |k|q\|q\alpha\|(4q^2+\ln q) \to 0,
\end{align*}
where again the last inequality follows by $t\geq \frac{1}{q^2}$, and the convergence is a direct consequence of~\eqref{eq:podciag}. 
Therefore for any $k\in\Z$ and $x\in A_q$
\begin{equation}\label{eq:unif}
\left|\hat{f}^{(q)}(x+k\|q\alpha\|)-\hat{f}_q(x) \right| \leq A(q) \text{, where }A(q) \to 0 \text{ for }q\to \infty. 
\end{equation}
Let $\overline{f}\colon (0,1) \to \R$ be given by the formula $\overline{f}(x)=\hat{f}(1-x)$. We have
\begin{equation*}
\overline{f}_q(1-x)=\hat{f}_q(x) \text{ and }\overline{f}^{(q)}(1-x)=\sum_{j=0}^{q-1}\hat{f}(x-j\alpha).
\end{equation*}
Since $\|q_n\alpha\|=\|q_n(1-\alpha)\|$, it follows by~\eqref{eq:podciag} that
\begin{equation*}
\lim_{k\to \infty} q_{n_k}^3 \|q_{n_k}(1-\alpha)\|=0.
\end{equation*}
Notice that the denominators $q_n$ in the continued fraction expansions of $\alpha$ and $1-\alpha$ are the same. Hence the sets $A_q$ are symmetric with respect to the point $1/2$, and the arguments used to justify~\eqref{eq:unif} can be applied also to the function $\overline{f}$. Therefore for any $k\in\Z$ we obtain that
\begin{equation}\label{eq:unif2}
\left|\overline{f}^{(q)}(x+k\|q\alpha\|)-\overline{f}_q(x) \right|\to 0 \text{ uniformly on }A_q.
\end{equation}

We will estimate now $\int_{[0,1)\setminus A_q} \left|\hat{f}^{(q)}(x+k\|q\alpha\|) - \hat{f}_q(x) \right|dx$. Notice that by~\eqref{eq:blisko}
\begin{equation*}
\lambda([0,1)\setminus A_q)=2q\left(\frac{1}{q^3}+|k|\|q\alpha\|\right)<2q\left(\frac{1}{q^3}+\frac{|k|}{q^3} \right)=\frac{2(|k|+1)}{q^2}.
\end{equation*}
Since $\hat{f}$ is a decreasing positive function and $\frac{2|k|+1}{q^2}<\frac{2}{q^2}$, we obtain (for $q$ such that $\frac{2(|k|+1)}{q^2}<e^{-1}$)
\begin{multline}\label{eq:calka}
\int_{[0,1)\setminus A_q} |\hat{f}^{(q)}(x+k\|q\alpha\|)- \hat{f}_q(x)|\ dx \leq 2q\int_{0}^{\frac{2(|k|+1)}{q^2}} \hat{f}(x)-1\ dx\\
=2q\left(-x\ln x\right)\big|_{x=0}^{\frac{2(|k|+1)}{q^2}}<-\frac{4}{q}\ln{\frac{2}{q^2}} \to 0 \text{ for }q \to \infty.
\end{multline}
In a similar way,
\begin{equation}\label{eq:calka2}
\int_{[0,1)\setminus A_q} |\overline{f}^{(q)}(x+k\|q\alpha\|)- \overline{f}_q(x)|dx \to 0 \text{ for }q \to \infty.
\end{equation}
Combining~\eqref{eq:unif},~\eqref{eq:unif2},~\eqref{eq:calka} and~\eqref{eq:calka2}, we obtain
\begin{equation*}
\int_{[0,1)}  |{f}^{(q)}(x+k\|q\alpha\|)- {f}_q(x)|dx \to 0 \text{ for }q \to \infty.
\end{equation*}

For $m>0$ we have
\begin{multline*}
\int_{[0,1)} \left|f^{(mq)}(x+k\|q\alpha\|)-mf_q(x) \right| dx\\
\leq \sum_{k=0}^{m-1} \int_{[0,1)} \left| f^{(q)}(x+k\|q\alpha\|)-f_q(x) \right| dx \to 0 \text{ for }q\to\infty.
\end{multline*}
Since for $m\in\Z$  $f^{(-mq)}(x)=-f^{(mq)}(x+m\|q\alpha\|)$, therefore for $m<0$ we obtain
\begin{multline*}
\int_{[0,1)} \left|f^{(mq)}(x)-mf_q(x) \right| dx \\
= \int_{[0,1)} \left|-f^{(-mq)}(x+m\|q\alpha\|)-mf_q(x) \right| dx \\
= \int_{[0,1)} \left|f^{(-mq)}(x+m\|q\alpha\|)-(-m)f_q(x) \right| dx \\
\leq \sum_{k=0}^{-m-1} \int_{[0,1)} \left| f^{(q)}(x+m\|q\alpha\|+k\|q\alpha\|)-f_q(x) \right| dx \to 0 \text{ for }q\to\infty.
\end{multline*}
Since the function $mf_q$ has the same distribution as $m\tilde{f}_q$ (the function $\tilde{f}_q$  was defined in~\eqref{eq:falka}), therefore it suffices to use Lemma~\ref{lm:zbieznosc} to deduce that
\begin{equation*}
\left(f^{(mq_{n_k})}\right)_{\ast}\lambda \to \nu_m.
\end{equation*}
This ends the proof.
\end{proof}

By the proof of Theorem 2.7 in~\cite{almn}\footnote{This theorem is a version of the Koksma inequality in the Banach space of functions whose Fourier coefficients are of order $\textrm{O} \left(\frac{1}{n}\right)$.} it follows that the sequence $\left(f^{(q_n)}\right)_{n\in\N}$ is bounded in $L^2(\T,\lambda)$. Therefore Proposition~\ref{pr:l+p} and Lemma~\ref{lm:lemma5} yield the following corollary.
\begin{wn}
Let $m\in \Z\setminus \{0\}$ and $\alpha\in\R\setminus\Q$ be such that
\begin{equation}
\liminf_{n \to \infty} q_n^3 \||q_n \alpha\| =0.
\end{equation}
Then for some $t_n \to \infty$
\begin{equation*}
U_{t_n} \to \int_{\R}U_{T_t} d\nu_m(t).
\end{equation*}
\end{wn}

\begin{uw}
It is easy to see that the measures $\nu_m$ are absolutely continuous: $\nu_m=(M_m \circ (\ln\frac{1}{2 \sin \pi \cdot}))_{\ast}(\lambda)$ and $M_m \circ (\ln\frac{1}{2 \sin \pi \cdot})$ has two inverse branches which are both absolutely continuous. The densities of the measures $\nu_m$ will be calculated in the next section.
\end{uw}

\subsubsection{Densities of limit distributions}
Let us first introduce the necessary notation. Let $\oh\colon\R\to\R$ be given by
\begin{equation*}
\oh(x)= \left\{
	\begin{array}{rl}
		\frac{1}{\sqrt{e^{2x}-1}} & \text{ for } x>0\\
		0 & \text{ for } x\leq 0.
	\end{array}
\right.
\end{equation*}
Let
\begin{equation*} 
h(x) = \left\{
\begin{array}{rl} \frac{2}{\pi}(4e^{2x}-1)^{-\frac{1}{2}} & \text{ for } x  > -\ln 2\\
0 & \text{ for } x \leq -\ln 2 \end{array} \right.
\end{equation*}
and for $t\neq 0$ let $h_t(x)=\frac{1}{|t|}h\left(\frac{x}{t}\right)$.

\begin{uw}\label{lm:gest}
Notice that given an absolutely continuous measure $\mu$ on $\R$ for $t\neq 0$ and $f\in L^1(\R,(M_t)_{\ast}(\mu))$ we have
\begin{multline*}
\int_{\R} f(x) d (M_t)_{\ast}(\mu)(x)\\
= \int_{\R} f(tx) d \mu(x)=\int_{\R}f(tx)\frac{d\mu}{d\lambda}(x)dx=\int_{\R}f(y)\frac{1}{|t|}\frac{d\mu}{d\lambda}\left(\frac{y}{t} \right)dy.
\end{multline*}
Therefore the density of $(M_t)_{\ast}(\mu)$ for $t\neq 0$ is given by $\frac{1}{|t|}\frac{d\mu}{d\lambda}\left(\frac{\cdot}{t}\right)$.
\end{uw}

\begin{lm}\label{lm:lemma7}
The density of the measure $\nu_t$ for $t\neq 0$\footnote{The measure $\nu_t$ is defined on page~\pageref{eq:nu}, see~\eqref{eq:nu}.} is given by $h_t$.
\end{lm}
\begin{proof}
We will show first that $h_1=h$ is the desity of the measure $\nu$. Indeed, for $b>a>-\ln 2$ we have
\begin{multline*}
\left(\left(\ln \frac{1}{2\sin \pi \cdot} \right)_{\ast} \lambda \right)(a,b)= \lambda\left( \{x\in (0,1) \colon \ln 2 \sin \pi x \in (-b,-a)\}\right)\\
= \lambda \left(\left\{ x\in (0,1) \colon \sin \pi x \in \left(\frac{e^{-b}}{2},\frac{e^{-a}}{2} \right)   \right\}\right)\\
= 2\lambda \left(\frac{1}{\pi}\arcsin \frac{e^{-b}}{2}, \frac{1}{\pi}\arcsin \frac{e^{-a}}{2}\right)=\int_a^b \frac{2}{ \pi}(4e^{2x}-1)^{-\frac{1}{2}}\ dx,
\end{multline*}
since the derivative of the function $2\arcsin (\frac{1}{2}e^{-x})$ is equal to $-\frac{2}{\pi}(4e^{2x}-1)^{-\frac{1}{2}} $. The claim for $t=1$ follows by the following inequality for $x \in (0,1)$:
\begin{equation*}
\ln\frac{1}{2\sin \pi x}\geq\ln \frac{1}{2\sin\frac{\pi}{2}}=-\ln 2.
\end{equation*}
To complete the proof it suffices to use Remark~\ref{lm:gest}.
\end{proof}

Let the function $v\colon (0,\infty)\to \R$ be given by the formula $v(x)=\sqrt{\frac{x}{e^{2x}-1}}$.

\begin{lm}\label{lm:funk}
The function $v$ is analytic on the interval $[0,+\infty)$.
\end{lm}

\begin{proof}
Notice that the function $(0,+\infty) \ni x \mapsto \frac{x}{e^{2x}-1} \in \R$ is analytic. Moreover, for $x>0$ we have $\frac{x}{e^{2x}-1}>0$ and 
\begin{equation}\label{grdwa}
\lim_{x\to 0} \frac{x}{e^{2x}-1}=\frac{1}{2}.
\end{equation} 
Using the fact that the square root has an analytic branch in $\C \setminus \{z\in\C \colon \text{Re}(z)<0 \text{ and }\text{Im}(z)=0 \}$, we conclude that also the function $v$  is analytic. By~\eqref{grdwa}, the function $v$ can be extended in an analytic way to the interval $(-\varepsilon,+\infty)$ for $\varepsilon>0$ sufficiently small. This ends the proof.
\end{proof}

Denote by $a_n$ the coefficients in the Taylor series expansion of $v$ around zero. This means that for $x>0$ small enough we have
\begin{equation*}
v(x)=\sqrt{\frac{x}{e^{2x}-1}}=\sum_{n=0}^{\infty}a_n x^n.
\end{equation*}
Calculating the consecutive derivatives of $v$, we can easily find a few first terms of this expansion. We have
\begin{equation*}
a_0=\frac{\sqrt{2}}{2},\ a_1=-\frac{\sqrt{2}}{4},\ a_2=\frac{\sqrt{2}}{48},\ a_3=\frac{\sqrt{2}}{96}.
\end{equation*}

\begin{lm}\label{lm:qq}
For any $k\in\N$
\begin{equation}\label{eq:pocho}
\frac{d^n}{dx^n}(\overline{h})(x)=\frac{\sum_{k=0}^{n}w_{k}e^{2kx}}{(e^{2x}-1)^{n+\frac{1}{2}}}
\end{equation}
for some $w_k\in\R$ dependent on $n$.
\end{lm}

\begin{proof}
The proof will be inductive with respect to $n$. For $n=0$ the formula~\eqref{eq:pocho} clearly holds. Suppose that this formula holds for some $n\in \N$. Notice that for any $0\leq k\leq n$ we have
\begin{multline*}
\frac{d}{dx}\left(\frac{e^{2kx}}{(e^{2x}-1)^{n+\frac{1}{2}}}\right)\\
=\frac{2ke^{2kx}(e^{2x}-1)^{n+\frac{1}{2}}-(n+\frac{1}{2})(e^{2x}-1)^{n-\frac{1}{2}}2e^{2x}e^{2kx}}{(e^{2x}-1)^{2n+1}}\\
=\frac{2ke^{2kx}(e^{2x}-1)-2e^{2x}e^{2kx}(k+\frac{1}{2})}{(e^{2x}-1)^{n+\frac{3}{2}}},
\end{multline*}
which ends the proof.
\end{proof}

\begin{lm}\label{lm:wlasnoscw}
For any $n\in \N$ the condition $\mathcal{W}\left(x\mapsto x^n \cdot\frac{d^n}{dx^n} \oh(x)\right)$ holds.
\end{lm}
\begin{proof}
Using Lemma~\ref{lm:qq}, we obtain
\begin{equation*}
\left|\frac{d^n}{dx^n}(\overline{h})(x)\right|=\frac{\sum_{k=0}^{n}|w_{k}|\cdot e^{2kx}}{(e^{2x}-1)^{n+\frac{1}{2}}}\leq \frac{\sum_{k=0}^{n}|w_{k}|\cdot e^{2nx}}{(e^{2x}-1)^{n+\frac{1}{2}}}.
\end{equation*}
Notice that
$$
\lim_{x\to 0^+} x^{1/2}\cdot\frac{x^{n}\cdot e^{2nx}}{(e^{2x}-1)^{n+1/2}}=2^{-n-1/2},
$$
$$
\lim_{x\to \infty}e^{x/2}\cdot \frac{x^n\cdot e^{2nx}}{(e^{2x}-1)^{n+1/2}} =0
$$
and function $\oh$ is smooth. Therefore there exist $t>0$ and $A>1$ such that for $x> t$ we have 
$$
\left|x^n\cdot \frac{d^n}{dx^n}\oh(x) \right| < A e^{-x/2}
$$
and for $0<x<t$ we have
$$
\left|x^n\cdot \frac{d^n}{dx^n}\oh(x) \right|< A x^{-1/2}.
$$
The claim follows.
\end{proof}

\begin{lm}\label{lm:kod}
Let $\sum_{n=0}^{\infty}a_nx^n$ be the Taylor series expansion around zero of the function $v\colon (0,\infty)\to\R$ given by the formula
$$
v(x)=\sqrt{\frac{x}{e^{2x}-1}}.
$$
Then the functions $c_n$ of the variables $t_1,t_2,t_3$ given by
\begin{equation}\label{eq:cn}
c_n(t_1,t_2,t_3)=\sum_{k_1+k_2+k_3=n}\prod_{i=1}^{3}\frac{a_{k_i}\Gamma(k_i+1/2)}{t_i^{k_i}},
\end{equation}
distinguish points from the lines of the form $t_1+t_2+t_3=c$ for any $c\in\R$ for the set of full measure $\mu^{\otimes 3}$ for any continuous measure $\mu \in \mathcal{P}(\R)$.
\end{lm}

\begin{proof}
\footnote{Maple 9.5. was used to make the necessary calculations. The source code is included in Section~\ref{se:maple}.}
Fix $c\in R$ and take $t_1,t_2,t_3\in\R\setminus\{0\}$ such that $t_1+t_2+t_3=c$. Let $x_i=1/t_i$ for $1\leq i\leq 3$.

The expressions $c_n$ for $1\leq n\leq 5$ given by the formula~\eqref{eq:cn} are (using the notation from Section~\ref{se:ws}) of the following form:
\begin{align*}
c_1=&a_1\cdot a_0^2 \cdot \Gamma(3/2)\cdot (\Gamma(1/2))^2 \cdot m_1=-\frac{\sqrt{2}}{16}\pi^{3/2}\cdot m_1,\\
c_2=&\frac{\sqrt{2}}{128}\pi^{3/2}\cdot m_2 + \frac{\sqrt{2}}{64}\pi^{3/2}\cdot m_{1,1}=\frac{\sqrt{2}}{128}\pi^{3/2}\left(m_1 \right)^2,\\
c_3=&\frac{5\sqrt{2}}{512}\pi^{3/2}\cdot m_3-\frac{\sqrt{2}}{512}\pi^{3/2}\cdot m_{2,1}-\frac{\sqrt{2}}{256}\pi^{3/2}\cdot m_{1,1,1},\\
c_4=&\frac{\sqrt{2}}{2048}\pi^{3/2}\cdot m_{2,1,1}+ \frac{\sqrt{2}}{4096}\pi^{3/2}\cdot m_{2,2}\\
&-\frac{5\sqrt{2}}{2048}\cdot m_{3,1}-\frac{21\sqrt{2}}{8192}\pi^{3/2}\cdot m_4,\\
c_5=&-\frac{\sqrt{2}}{16384}\cdot \pi^{3/2}\cdot m_{2,2,1}+\frac{5\sqrt{2}}{8192}\pi^{3/2}\cdot m_{3,1,1}\\
&+\frac{5\sqrt{2}}{16384}\pi^{3/2}\cdot m_{3,2}+\frac{21\sqrt{2}}{32768}\pi^{3/2}\cdot m_{4,1}\\
&-\frac{399\sqrt{2}}{32768}\pi^{3/2}\cdot m_5.
\end{align*}
Notice that the value of $c_2$ doesn't not give us any new information when we know $c_1$. The same applies to $c_4$ (the calculations needed to check this are slightly longer).

If we know the values of $c_1,c_3$ and $c_5$, we also know the values of the following expressions:
\begin{align}
& m_1= -\frac{16}{\sqrt{2}\cdot \pi^{3/2}}c_1,\label{d}\\
& 5m_3-m_{2,1}-2m_{1,1,1}=\frac{512}{\sqrt{2}\cdot\pi^{3/2}}c_3,\label{dd}\\
& -399m_5+21m_{4,1}+10m_{3,2}+20m_{3,1,1}-2m_{2,2,1}=\frac{32768}{\sqrt{2}\cdot \pi^{3/2}}c_5.\label{ddd}
\end{align}
Moreover (directly by~\eqref{d}), we know the values of
\begin{align}
(m_1)^2&=m_2+2m_{1,1},\label{d2}\\
(m_1)^3&=m_3+3m_{2,1}+6m_{1,1,1},\label{dddd}\\
(m_1)^5&=m_5+5m_{4,1}+10m_{3,2}+20m_{3,1,1}+30m_{2,2,1}.\label{ddddd}
\end{align}
Using~\eqref{dd} and~\eqref{dddd}, we calculate
$$
 m_3\text{ and }m_{2,1}+2m_{1,1,1}.
$$
Similarly,~\eqref{ddd} and~\eqref{ddddd} give us the value of
\begin{equation*}
-\frac{32768}{\sqrt{2}\cdot \pi^{3/2}}c_5 - 399 (m_1)^5=-32 (m_{2,1}+2m_{1,1,1})(63m_2+62m_{1,1}).
\end{equation*}
Since $m_{2,1}+2m_{1,1,1}=(x_1+x_2)(x_1+x_3)(x_2+x_3)$, and the measure $\mu$ is continuous, without loss of generality we may assume that $m_{2,1}+2m_{1,1,1}\neq 0$. Therefore we know the value of
\begin{equation}\label{d3}
63m_2+62m_{1,1}.
\end{equation}
Using~\eqref{d2} and~\eqref{d3} we can therefore calculate $m_2$. The values of $m_1,m_2$ and $m_3$ determine uniquely the set $\{x_1,x_2,x_3\}$, which ends the proof.
\end{proof}

\subsubsection{Proof of the main theorem}
\begin{proof}[Proof of Theorem~\ref{tw:5.19}]
We will use Proposition~\ref{tw:nowametoda} (see page~\pageref{tw:nowametoda}). We will show that all its assumptions for $d=3$ are fulfilled. We check that:
\begin{itemize}
\item
The considered flow is weakly mixing~\cite{MR2063628}. 
\item
As the measure $P$ appearing in the assumptions of Proposition~\ref{tw:nowametoda}, by Lemma~\ref{lm:lemma5} (page~\pageref{lm:lemma5}) and Proposition~\ref{pr:l+p} (page~\pageref{pr:l+p}) we can take the measure $\nu$ given by the formula~\eqref{eq:miaranu}, i.e.
\begin{equation*}
\nu=\left(\ln\left(\frac{1}{2\sin \pi \cdot} \right)\right)_{\ast}(\lambda)
\end{equation*}
(the assumptions of Proposition~\ref{pr:l+p} are fulfilled, in particular, the boundedness of the sequence $\left(f_0^{(q_n)}\right)_{n\in\N}$ is connected with the absence of mixing, see~\cite{Kochergin76, MR2336903}).
\item
The density $h$ of the measure $\nu$ satisfies the conditions $(i)-(iii)$ of Proposition~\ref{tw:nowametoda} by Lemma~\ref{lm:lemma7} and Lemma~\ref{lm:funk} (page~\pageref{lm:funk}). Moreover, the condition $(iv)$ is also satisfied by Lemma~\ref{lm:wlasnoscw} (page~\pageref{lm:wlasnoscw}).
\item
The last assumption on the distinguishing points from lines of the form ine $x_1+\dots+x_n=c$ is also satisfied by Lemma~\ref{lm:kod} (page~\pageref{lm:kod}).
\end{itemize}
\end{proof}

\subsubsection{Appendix: Source code in Maple 9.5 used for needed expressions in Lemma~\ref{lm:kod}}\label{se:maple}
\footnotesize
\begin{verbatim}
> restart:
> with(combinat):
\end{verbatim}
\begin{verbatim}
> for i from 0 to 20 do 
>   a[i]:=coeff(convert(series(sqrt(x/(exp(2*x)-1)),x=0,21),polynom),x,i):
>   b[i]:=a[i]*GAMMA(i+1/2): 
> od:
> 
> SUM:=proc(set)
> local Su,j:
> Su:=0;
> for j from 1 to nops(set) do Su:=Su+set[j] od:
> end proc:
> 
> POLY:=proc(numb,List)::polynom: 
> local w,resul,R,i,Se,V,sum,j,counter,u:
> Se:={}:
> if nops(List)<=numb then
>   for i from 1 to numb do Se:=Se union {x[i]} od:
> end if:
> R:=choose(Se,nops(List));
> for i from 1 to nops(R) do
>   V:=permute(R[i]):
>   sum[i]:={};
>   for j from 1 to nops(V) do
>     counter:=1:
>     for u from 1 to nops(V[j]) do
>       counter:=counter*(V[j][u]^List[u]):
>     od:
>     sum[i]:=sum[i] union {counter}:
>     w[i]:=SUM(sum[i]):    
>   od:
> od:
> resul:=0;
> for i from 1 to nops(R) do resul:=resul+w[i]:od:
> end proc:
> 
> GEN:=proc(p,n)
> local i,j,RESULT;global K,C:
> K:=[]:for i from 1 to nops(partition(n)) do 
>     if nops(partition(n)[i])<=p then K:=[op(K),partition(n)[i]] 
>     fi;
> od;
> K;
> for i from 1 to nops(K) do
>   C[i]:=b[0]^(p-nops(K[i])):
>   for j from 1 to nops(K[i]) do C[i]:=C[i]*b[K[i][j]]:od:
> od:
> RESULT:=0;
> for i from 1 to nops(K) do
>   RESULT:=RESULT+C[i]*POLY(p,K[i]):
>   od:
> end proc:
> 
> B[1]:=GEN(3,1);
\end{verbatim}
\begin{equation*}
B[1]:=-\frac{1}{16}\pi^{3/2}\sqrt{2} (x_1+x_2+x_3)
\end{equation*}
\begin{verbatim}
> B[2]:=GEN(3,2);
\end{verbatim}
\begin{equation*}
B[2]:=\frac{1}{64}\pi^{3/2}\sqrt{2}(x_1x_2+x_1x_3+x_2x_3)+\frac{1}{128}\pi^{3/2}\sqrt{2}(x_1^2+x_2^2+x_3^2)
\end{equation*}
\begin{verbatim}
> B[3]:=GEN(3,3);
\end{verbatim}
\begin{align*}
B[3] :=& -\frac{1}{256}\pi^{3/2}\sqrt{2}x_1x_2x_3-\frac{1}{512}\pi^{3/2}\sqrt{2}(x_1x_2^2+x_2x_1^2\\
+&x_1x_3^2+x_3x_1^2+x_2x_3^2+x_3x_2^2)+\frac{5}{512}\pi^{3/2}\sqrt{2}(x_1^3+x_2^3+x_3^3)
\end{align*}
\begin{verbatim}
> B[4]:=GEN(3,4);
\end{verbatim}
\begin{align*}
B[4]:=&\frac{1}{2048}\pi^{3/2}\sqrt{2}(x_1x_2^2x_3+x_2^2x_3x_1+x_1^2x_3x_2)+\frac{1}{4096}\pi^{3/2}(x_1^2x_2^2+x_1^2x_3^2+x_2^2x_3^2)\\
-&\frac{5}{2048}\pi^{3/2}\sqrt{2}(x_1^3x_2+x_2^3x_1+x_1^3x_3+x_3^3x_1+x_2^3x_3+x_3^3x_2)\\
-&\frac{21}{8192}\pi^{3/2}\sqrt{2}(x_1^4+x_2^4+x_3^4)
\end{align*}
\begin{verbatim}
> B[5]:=GEN(3,5);
\end{verbatim}
\begin{align*}
B[5]:=&-\frac{1}{16384}\pi^{3/2}\sqrt{2}(x_3x_1^2x_2^2+x_1x_2^2x_3^2+x_2x_1^2x_3^2)\\
+&\frac{5}{8192}\pi^{3/2}\sqrt{2}(x_2x_3x_1^3+x_1x_3x_2^3+x_1x_2x_3^3)\\
+&\frac{5}{16384}\pi^{3/2}\sqrt{2}(x_1^2x_2^3+x_2^2x_1^3+x_1^2x_3^3+x_3^2x_1^3+x_2^2x_3^3+x_3^2x_2^3)\\
+&\frac{21}{32768}\pi^{3/2}\sqrt{2}(x_1x_2^4+x_2x_1^4+x_1x_3^4+x_3x_1^4+x_2x_3^4+x_3x_2^4)\\
-&\frac{399}{32768}\pi^{3/2}\sqrt{2}(x_1^5+x_2^5+x_3^5)
\end{align*}
\begin{verbatim}
> expand(POLY(3,[1])^2);
\end{verbatim}
\begin{equation*}
x_1^2+2x_1x_2+2x_1x_3+x_2^2+2x_2x_3+x_3^2
\end{equation*}
\begin{verbatim}
> expand(POLY(3,[1])^3);
\end{verbatim}
\begin{equation*}
x_1^3+3x_2x_1^2+3x_3x_1^2+3x_1x_2^2+6x_1x_2x_3+3x_1x_3^2+x_2^3+3x_3x_2^2+3x_2x_3^2+x_3^3
\end{equation*}
\begin{verbatim}
> expand(POLY(3,[1])^5);
\end{verbatim}
\begin{multline*}
20x_2x_3x_1^3+20x_1x_3x_2^3+30x_3x_1^2x_2^2+20x_1x_2x_3^3+30x_1x_2^2x_3^2+30x_2x_1^2x_3^2\\
+10x_1^2x_2^3+10x_2^2x_1^3+10x_1^2x_3^3+10x_3^2x_1^3+10x_2^2x_3^3+10x_3^2x_2^3\\
+5x_1x_2^4+5x_2x_1^4+5x_1x_3^4+5x_3x_1^4+5x_2x_3^4+5x_3x_2^4+x_1^5+x_2^5+x_3^5
\end{multline*}
\begin{verbatim}
> factor((-(32768)/(sqrt(2)Pi^(3/2)))B[5]-399*POLY(3,[1])^5);
\end{verbatim}
\begin{equation*}
-32(x_2+x_3)(x_1+x_3)(x_1+x_2)(63x_1^2+62x_1x_2+62x_1x_3+63x_2^2+62x_2x_3+63x_3^2)
\end{equation*}
\begin{verbatim}
> factor(POLY(3,[2,1])+2*POLY(3,[1,1,1]));
\end{verbatim}
\begin{equation*}
              (x_2 + x_3) (x_1 + x_3) (x_1 + x_2)
\end{equation*}
\normalsize

\footnotesize{
\bibliography{cala.bib}
}
\normalsize
\end{document}